\documentclass{article} 

\usepackage{amsmath,amssymb,amsthm,bbm,srcltx,graphicx,a4wide,xr,xargs, mathtools}
\usepackage[shortlabels]{enumitem}

\usepackage[usenames,dvipsnames]{xcolor}
\usepackage[colorlinks]{hyperref}
\hypersetup{citecolor=ForestGreen}

\usepackage{cleveref}
\numberwithin{equation}{section}

\newcommand{\toi}{\to\infty}
\newcommand{\eind}{\stackrel{d}{=}}

\newcommand{\dto}{%
	\mathrel{\vbox{\offinterlineskip\ialign{%
				\hfil##\hfil\cr
				$\scriptstyle d$\cr
				$\longrightarrow$\cr
			}}}}

\newcommand{\asto}{%
	\mathrel{\vbox{\offinterlineskip\ialign{%
				\hfil##\hfil\cr
				$\scriptstyle a.s.$\cr
				$\longrightarrow$\cr
}}}}

\newcommand{\EE}{\mathbb{E}}
\newcommand{\PP}{\mathbb{P}}

\newcommand\1[1]{\mathbbm{1}_{#1}}

\newcommand{\dsum}{\displaystyle\sum}
\newcommand{\tsum}{\textstyle\sum}
\newcommand{\dint}{\displaystyle\int}

\newcommand{\NN}{\mathbb{N}}

\newcommand{\XX}{\mathbb{X}}
 
\newcommand{\bbS}{\mathbb{S}}

\newcommand{\bex}{\boldsymbol{e}}

\newcommand{\PPP}{\mathrm{PPP}}
\newcommand{\RAM}{\mathrm{RAM}}
\newcommand{\Npp}{N}
\newcommand{\nmap}{\mathcal{V}}

\newcommand{\sB}{\mathcal{B}}

\newcommand{\sS}{\mathcal{X}}

\newcommand{\tgamma}{\textit{gamma}}
\newcommand{\tbeta}{\textit{beta}}
\newcommand{\texp}{\textit{Exp}}
\newcommand{\tunif}{\textit{U}}

\newcommand{\qand}{\quad \mbox{ and } \quad}
\newcommand{\vep}{\varepsilon}

\newtheorem{theorem}{Theorem}[section]
\newtheorem{lemma}[theorem]{Lemma}
\newtheorem{definition}[theorem]{Definition}
\newtheorem{corollary}[theorem]{Corollary}
\newtheorem{proposition}[theorem]{Proposition}
\newtheorem{hypothesis}[theorem]{Assumption}
\theoremstyle{remark}
\newtheorem{remark}[theorem]{Remark}
\newtheorem{example}[theorem]{Example}

\newcommand{\COM}[1]{} 
\newcommand{\sep}{\mbox{$\bullet$\ }}

\newcommand{\Sub}{\tau}
\newcommand{\Inv}{L}

\usepackage{csquotes}

\title{On generalized arcsine laws and residual allocation models  }
\author{Bojan Basrak}

\date{%
{Department of Mathematics, Faculty of Science}\\
	{University of Zagreb}\\ 
}

\begin{document}

\maketitle

\begin{abstract}
Based on their earlier studies of the arcsine law, Pitman and Yor in \cite{PY97} constructed a widely adopted  PD($\alpha, \theta)$ family of random mass-partitions with parameters $\alpha \in [0,1),\ \theta+\alpha>0$.  We propose an alternative model based on generalized perpetuities, which extends the PD family in a continuous manner,  
incorporating  any $\alpha\geq 0$.  This perspective yields a new, concise proof  for the stick-breaking (or residual allocation) representations of  PD distributions, recovering the classical results of McCloskey and Perman in particular.  We apply this framework to provide a  constructive and intuitive proof of  Pitman and Yor's generalized arcsine law concerning the partitions arising from $\alpha$-stable subordinators for $\alpha \in (0,1)$. The result shows that the random partitions generated by stable subordinators have  identical distributions when observed over temporal or spatial  intervals. This theorem has a number of significant implications for excursion theory.  As a corollary, using purely probabilistic arguments, we obtain general arcsine laws for excursions of $d$-dimensional Bessel process for $0<d<2$, and Brownian motion in particular.
\end{abstract}


\section{Introduction}
Random discrete distributions are typically represented in the form $ \sum_i V_i \delta_{Q_i}$ where $(V_i)$  denotes  a sequence of nonnegative random variables which sum up to 1 and $(Q_i)$ represents
a sequence of (random) elements  with values in a chosen state space. The components are often assumed to be independent and  $(Q_i)$ is assumed to be an i.i.d. sequence, see \cite{IJ01}.
With applications spanning many areas, random discrete distributions play a particularly significant role in genetics \cite{Fe10book,Cr16}, excursion theory~\cite{PPY92, Pe93}, combinatorial stochastics \cite{AT94,Pi06book}, and Bayesian nonparametrics and machine learning \cite{ALC19,DeBetal13}.

%
The random weights $V_i$ can be represented as a random element $\sum_i \delta_{V_i}$ of the space of counting measures in $[0,\infty)$  whose points add up to one exactly,  denoted by  $M_{p,1}$. We  refer to $\sum_i \delta_{V_i}$  as a (proper) random {\em mass-partition} (see Bertoin \cite{Be06book}) or, somewhat imprecisely, {\em random distribution}. We sometimes identify $\sum_i \delta_{V_i}$  with the countable set $\{V_i\}$ or
 its nonincreasing ordering and work with nonnegative
sequences $V_1 \geq V_2\geq V_3 \geq \ldots$ which add up to one, as done by \cite{PY97} or \cite{Pi06book},  however the advantage of using the point process representation will become apparent in the sequel. 
If there exists a sequence $(V'_i)$ such that $\sum_i \delta_{V_i} \eind \sum_i \delta_{V'_i}$, we will call it a representation of the random distribution $\sum_i \delta_{V_i}$.

Among numerous models for random mass-partitions,  the most widely adopted is the Poisson-Dirichlet PD($\alpha,\theta$) family with parameters $\alpha \in [0,1),\, \theta >-\alpha$,  due to Pitman and Yor \cite{PY97}. It elegantly extrapolates  the standard Poisson-Dirichlet  and standard Pitman-Yor distribution, which correspond to the cases PD($0,\theta$) and PD($\alpha,0$) respectively, which were originally suggested by Kingman in \cite{Ki75}. These two special PD models were extensively studied and applied 
	because of their role in stochastic processes, combinatorics,  number theory, statistical physics and many other fields, see for instance  \cite{Bi13,Pi06book,Be06book,De81}.
Over the years, several extensions of the PD model were proposed, e.g. see  \cite{DeBetal13,Ja13,DLS23,RPS25}. However, we could not find a model that suggests how to naturally extend  the PD model for $\alpha\geq 1$, as we will do here using the concept of {\em generalized beta/gamma perpetuities}. 


Clearly, any sequence of nonnegative random variables $(X_i)$ satisfying $\sum_i X_i <\infty$ a.s., after  self-normalization gives rise to a random  mass-partition $\sum_i \delta_{V_i}$. Nevertheless, the standard random distributions used in the literature  arise in three canonical ways we describe next.

\paragraph{Self normalized point processes and subordinators.} 
  Consider  a subordinator  $(\Sub_t)$ with no drift and  with a L\'evy measure $\mu$ concentrated on $[0,\infty)$ which satisfies $\int x\wedge 1 \mu(dx) <\infty$. From $(\Sub_t)$, we obtain a whole family of random discrete distributions, simply by normalizing its jumps over a given time interval $[0,t]$, $t>0$.  Any such subordinator has a representation 
  \begin{equation}\label{eq:GenSub}
  	\Sub_t= \int x  \1{s \leq t} \Npp(ds,dx)
  \end{equation}
 where $\Npp$ is a Poisson point processes (PPP) with intensity measure $Leb \times \mu$, which we denote by 
\begin{equation}\label{eq:PPPgen}
	 \Npp  
 = \dsum_{i=1}^\infty \delta_{T_i,P_i}\sim \PPP(Leb \times \mu)\,.
\end{equation}

The space of counting measures on the background space $\XX$ is denoted by $M_p(\XX) $ in the sequel.  In our case $\XX$ is  typically equal to $[0,\infty)^d$ for some integer $d \geq 1$.
For convenience, we  write both $\{ x_i : i \in I \}$  and $\sum_{i \in I } \delta_{x_i}$ to denote
an element of $M_p(\XX) $. By $M_{p,f}$ and $M_{p,1}$ we denote the spaces of counting measures in $M_p[0,\infty)$ whose points sum up to a finite number, or to 1, respectively, see Section \ref{Sec:aux}.
To obtain random mass-partition  from a counting measure in $M_{p,f}$, we apply  the (self-normalization) map $\nmap:M_{p,f}\to M_{p,1}$ given by 
\begin{equation}\label{eq:nmapDef}
	\mu = \sum_i \delta_{x_i} \mapsto 
\nmap ( \mu) = \sum_i \delta_{x_i/\sum_j x_j} .
\end{equation}
From the point process $N$ in \eqref{eq:PPPgen} one can obtain a random  distribution for any $t>0$ by 
\begin{equation}\label{eq:VofNa}
 \sum_i \delta_{V_i}  =	\nmap\left(\sum_{T_i\leq t} \delta_{P_i} \right)\,.
\end{equation}

 \paragraph{Stick-breaking distributions.}
 To construct the second class of examples, assume that $(Y_i)$ denotes a sequence of random variables with values in $[0,1]$. Consider now the stick-breaking scheme
 \begin{equation}\label{eq:SBgenY}
 	V_1= Y_1,
 	\qand 
 	V_n = (1-Y_1)( 1- Y_2)\cdots  (1-Y_{n-1})Y_n,\ n \geq 2.
 \end{equation} 
 Observe that the sequence $(V_n)$  sums up to one  if and only if 
 $\prod_{j=1}^n(1-Y_j) \to 0$ as $n \toi$. 
 This holds 
almost surely if $(Y_j)$ is an i.i.d. sequence
 with nondegenerate distribution. 
 In any such case we say that $\sum_i \delta_{V_i}$ defines a general {\em stick-breaking distribution} or a general {\em residual allocation model} ($\RAM$).
  If $Y_j \sim \tbeta(a_j,b_j)$ for some strictly positive sequences $(a_j)$ and $(b_j)$,
 one arrives at a general infinitely-dimensional family
   $\mathcal{P}_\infty((a_j),(b_j))$
  models  suggested by Ishwaran and James in \cite{IJ01}. 
 
 It is known  that  all the  distributions in PD($\alpha,\theta$) family  have  representations as in  \eqref{eq:SBgenY}, see \cite{PPY92} or \cite{Pi06book}. We give a simple proof of this fact, by showing that the same holds for a larger family of distributions which arise from  beta/gamma perpetuities we describe next.
 
\paragraph{Generalized beta/gamma perpetuities.} 
The third important class of random distributions can be derived by normalizing randomly weighted 
sequence of  random variables $(G_i)$. It is useful in the sequel to abuse the notation and assume that $\nmap$ also maps a summable nonnegative sequence into a new sequence by self-normalization, that is for a sequence $(x_i)$ in $[0,\infty)^d$ or $[0,\infty)^\NN$ such that $\sum_i x_i< \infty$ we write
\begin{equation}\label{eq:nmapSeq}
	(x_n)\mapsto 
	\nmap ((x_n))  = \left( \frac{x_n}{\sum_i x_i}  \right)\,,
\end{equation}
which we also often tacitly consider  as an element of $M_{p,1}$ space.
 Recall, if $G_i\sim \tgamma(\alpha_i,1)$, $i =1,\ldots, n,$ is a finite sequence of independent random variables, the normalized sequence $	\nmap ( G_1 ,\ldots,G_n)$ has the Dirichlet distribution, which is clearly symmetric when $G_i$'s are i.i.d..
Since applying the self-normalization to an infinite sequence   of  nonnegative i.i.d. random variables $(G_i)$ is not an option,  
it was suggested  to consider a sequence 
 \begin{equation}\label{eq:GPin1}
	\nmap \left( G_n \Pi_n\right) =
	\nmap \left( (G_n \Pi_n)_n \right)
\end{equation}
where $(\Pi_n)$ denotes another (independent) sequence of nonnegative random weights which  makes
the products above summable a.s., cf.  the  concept of $P$-means in   \cite{Pi06book} or Dirichlet means in \cite{CR90} or \cite{JRY08}.
 There are  many choices for $(G_n)$ and $(\Pi_n)$, but some appear naturally in this context. Consider first
 \begin{equation}\label{eq:GPin}
	\nmap \left( G_n \Pi_n\right) \quad \mbox{with} \quad 
 	\Pi_n = \prod_{i=1}^{n-1} U_i, \ j \geq 1,
\end{equation}
for a sequence $(U_n)$ of independent nonnegative random variables independent of $(G_n)$. A well studied special case is obtained when $U_n$'s are
i.i.d. as well.  In that case, it is known that if $\EE \log U <0$ and if $\EE \log^+ G <\infty$, the sum $ X=  \sum_n G_n \Pi_n < \infty $ is a.s. finite and satisfies
$
X  \eind  U X +G\,,
$
where all three random variables on the right hand side are independent, see Vervaat \cite{Ve79}.
The random variables $X$ fitting this description are also  called  perpetuities \cite{EG94}. There exists extensive literature concerning  stochastic recursions of this type, see  \cite{DF99} and \cite{BDM16}.
In the case when $U_i$'s are merely independent we refer to  them as generalized perpetuities. 
The construction of  Dirichlet distribution through gamma random variables motivates the following.

\begin{definition}
We say that a random mass-partition $\sum_i \delta_{V_i}$
belongs to
{\em beta/gamma (perpetuity)} family of distributions  if
\begin{equation}\label{eq:GenGamPer}
\sum_i \delta_{V_i} \eind \nmap ( G_n \Pi_n) 
\end{equation}
where $	\Pi_n = \prod_{i=1}^{n-1} U_i,$ for
some independent  random variables $(G_i)$ and $(U_i)$, such that  $G_i$'s have identical gamma distribution  while $U_i$'s are beta distributed.
\end{definition}

 We will see  that all the  distributions in PD($\alpha,\theta$) family can be represented in this way.
 \COM{ In particular, the self-normalized jumps of an $\alpha$-stable subordinator}


\paragraph{Outline and main results.}
The main contributions of the paper can be summarized as follows:

i)  In \Cref{thm:GPi}, and a more general  \Cref{lem:gam}, we show that mass-partitions with beta/gamma
perpetuity representation 
belong to a class of stick-breaking schemes. The result covers all the    PD($\alpha,\theta$) distributions, and yields the classical results of McCloskey and Perman  as a corollary. 
ii)   Based on \Cref{thm:GPi}, we introduce a new, three parameter class of  residual allocation models which continuously extends the  family of PD distributions. Many alternative models have been suggested over the years,
 see for instance \cite{DeBetal13,Ja13,DLS23,RPS25}, but the advantage of our proposal is that it incorporates  any $\alpha\geq 0$
 in a  natural way.
In Section~\ref{sec:PYthm}, we provide a  constructive geometric proof of  Pitman and Yor's generalized arcsine law concerning the partitions arising from $\alpha$-stable subordinators for $\alpha \in (0,1)$,  see \Cref{thm:1} and \Cref{cor:AGammaV}.  Due to the connection with the   lengths of excursions of  Bessel processes of dimension $d\in (0,2)$, the theorem  has a pivotal role  in  theory of stochastic processes, see \cite{PPY92}.   
Consequently,  we present very general arcsine laws for excursions of $d$-dimensional Bessel process for $0<d<2$, and
 Brownian motion in particular.

We  also explore alternative representations for  PD($\alpha,\theta$)  distributions for  $\alpha \in (0,1)$ which turn out to be  useful  throughout the text. In particular, we show that all PD distributions have a representation in terms of a Poisson process $N^\vee$ on $(0,\infty)$ with intensity measure proportional to $ x^{-(\alpha+1)} e^{-x} dx$ for $\alpha \in (0,1)$. This yields a very short proof for one of the key results in \cite{PY97}, see Proposition 21 therein.

The main results of the article rely on two probabilistic ideas which have  not been previously applied  in this context. We use a new coupling argument to extend beta-gamma algebra to infinite sequences. This is the key behind 
the proof of  \Cref{lem:gam}, and consequently of \Cref{thm:GPi}.
On the other hand, the generalized arcsine law of Pitman-Yor, see  \Cref{thm:1},
is derived by  an intuitive and constructive thinning argument.
The rest of our methods  are essentially simple probabilistic arguments stemming from the basic properties of Poisson processes.
Some of the  prerequisites and auxiliary results
are postponed to Section~\ref{Sec:aux}. 

\section{Stick-breaking theorem}

It is useful in the sequel to let
$(\Gamma_i)$ denote the points of 
the unit rate homogeneous Poisson process  on $(0,\infty)$, that is $\Gamma_i = D_1 + \cdots + D_i$, $i \geq 1$ for an i.i.d. sequence of standard exponential random variables $(D_i)$. We abbreviate this by
\begin{equation}\label{eq:hPPP}
	\sum_{i=1}^\infty \delta_{\Gamma_i} \sim \PPP(Leb).
\end{equation}

It is also useful to introduce scaling operator  on $[0,\infty) \times M_p([0,\infty)^d)$, so that we can write 
for instance $ c \sum_i \delta_{ P_i} =  \sum_i \delta_{c P_i}$.

 \subsection{Size-biasing}
 One can use an i.i.d. sequence of standard exponential random variables $(E_n)$ to  
 independently mark each point of the Poisson process $\Npp= \sum_i \delta_{T_i,P_i}$ in \eqref{eq:PPPgen} to arrive at
 \begin{equation}\label{eq:NnaE}
 	N^E = \dsum_i  \delta_{T_i,P_i,E_i}\,.
 \end{equation}
 It is again a Poisson process with intensity measure equal to $Leb \times \mu \times \nu_{E}$ on $(0,\infty)^3$ where $\nu_E$  denotes the law of a generic $E$. 
   For each  $s>0$, one can transform $N^E$ into
  \begin{equation}\label{eq:NEkrozPgen}
  	N^{E/P} = N^{E/P}_s = \dsum_{T_i \leq s} \delta_{E_i/P_i, P_i}\,.
  \end{equation}
Note that $P_i$'s now denote only the jumps  of the corresponding subordinator until time $s$. This is again a Poisson point process, see \cite[Theorem 5.1]{LP18}, with  finitely many points in each set $[0,t] \times (0,\infty)$. Let $s$ be fixed and rename  the coordinates of $N^{E/P}$  as pairs $(Z_j,\tilde P_j)$,  so that $0 \leq  Z_1 \leq Z_2 \leq Z_3 < \ldots $, and write
  \begin{equation}\label{eq:NEkrozPZ}
	N^{E/P} 
	= \dsum_{i=1}^\infty \delta_{Z_i, \tilde P_i}\,.
\end{equation}
 Using the fact that $(E_i)$ is an i.i.d. sequence from exponential distribution, it is easy to see  that $(\tilde P_1, \tilde P_2, \tilde P_3, \ldots )$ simply correspond to the size biased permutation of $\{P_i : \, T_i \leq 1\}$ in \eqref{eq:NEkrozPgen}, see also Lemma 4.4 in \cite{PPY92}.
Therefore,  the self-normalized values $(P_i)$ can be represented in  size biased ordering as
\[
	(\tilde V_n)= \nmap(\tilde P_n)=
	\left(\frac{\tilde P_n}{\sum_j \tilde P_j} \right)\,.
\]

We explore next the effect of size-biasing on two standard models  initially suggested and studied by Kingman in \cite{Ki75}.

\paragraph{Standard Poisson-Dirichlet distribution.}
Assume that $	\Npp\sim \PPP(Leb \times\gamma)$ 
where $\gamma(dt) = t^{-1} e^{-t} dt $. The process $\Npp$ gives rise to the right continuous nondecreasing process $(\Sub_s)$ called gamma subordinator as in \eqref{eq:GenSub}.
The standard Poisson Dirichlet distribution is now defined simply as the self-normalization of $\sum_{T_t \leq a} \delta_{P_i}$ for a fixed $a>0$, that is as 
\begin{equation}\label{eq:PDa}
	\dsum_j \delta_{V_j}=	\nmap \left( \sum_{T_t \leq a} \delta_{P_i} \right)\,.
\end{equation}
As observed above
$N^{E/P}_a $
 is  a Poisson point process with  intensity measure $\nu^{(2)}$ say. 
It is easily seen, see \Cref{ex:EPab}, that
$\nu^{(2)} (du,dv) = \gamma_a (du) \otimes K(u,dv) $,  
where $\gamma_a (0,u) = \log (1+u)^a$ and $K(u,\cdot) $ is a probability marking kernel on $(0,\infty)\times \sB(0,\infty) $ corresponding to the  exponential distribution with parameter $1+u$.
This yields the following identity in distribution
\begin{equation}\label{eq:NEQ}
	N^{E/P}_a \eind \dsum_i \delta_{Z^a_i, G_i/(Z^a_i+1)}\,,
\end{equation}
where $(Z^a_i)_i$ represent the points of a Poisson process with intensity measure $\gamma_a$ which is independent of an i.i.d. sequence $(G_i)_i$ of standard exponential random variables.  It is simple to see that the sequence $(Z^a_i)_i$ has the distribution of $(\exp( \Gamma_i /a) -1)_i$ for a standard homogenous Poisson process $(\Gamma_i)$ in \eqref{eq:hPPP}. Recall that the size-biased sample $(\tilde V_j)$ can be derived by the self-normalization of the second coordinates of $	N^{E/P}_a$, therefore
\begin{align} \label{eq:tildVa}
	(\tilde V_n) =  \left( \frac{G_n /\exp( \Gamma_n/a) }{\sum_{j=1}^\infty G_j/\exp(\Gamma_j/a )} \right)
= \nmap (G_n \Pi_n),
\end{align}
where $\Pi_j = \prod_{i=1}^{j-1} \exp( (\Gamma_{i} - \Gamma_{i+1})/a)$,  by convention $\Pi_1=1$ in the sequel.
In particular, $\Pi_j$ is a product of i.i.d. random variables
from $\tbeta(a,1)$ distribution. Thus, the standard Poisson-Dirichlet distribution of $\sum_j \delta_{V_j}$ belongs to beta/gamma  family of distributions.
Using \eqref{eq:hPPP},  we can also write
\begin{equation}\label{eq:PDGamPoi}
	(\tilde V_n) =
	=  \nmap(G_n e^{ -\Gamma_n/a}) = 
\nmap(G_n e^{ -(D_1+\cdots + D_n)/a}).
\end{equation}


\paragraph{Standard Pitman-Yor distribution.}

In the seminal series of papers coauthored by Perman, Pitman and Yor \cite{PY92,PPY92,Pe93,PY97,PY97b,Pi97}, the authors showed a number of striking properties of the mass-partitions generated by stable  subordinators. In that setting, $\Npp \sim \PPP(Leb \times \mu_\alpha)$ where
$\mu_\alpha(u,\infty) = K u^{-\alpha}$ for any $u>0$, some $\alpha \in (0,1)$, and an arbitrary constant $K>0$.
The corresponding right continuous nondecreasing process $(\Sub_t)$ is 
the so called $\alpha$-stable subordinator (see Bertoin~\cite{Be99} or Revuz and Yor~\cite[Ch. III]{RY99}).
It is straightforward to see that point process $\Npp$ has the following scaling property for any fixed $c>0$
\begin{equation}\label{eq:ScalProp}
	\Npp\eind \dsum_i \delta_{cT_i,c^{1/\alpha} P_i} \,.    
\end{equation} 
In particular, for any $s>0$
\begin{equation}\label{eq:VofSeq1}
	\dsum_j \delta_{V_j}= \nmap\left( \sum_{T_i\leq 1} \delta_{ P_i}\right) \eind	
\nmap \left( \sum_{T_i\leq s} \delta_{s^{-1/\alpha} P_i}\right) =\nmap \left(\sum_{T_i\leq s} \delta_{ P_i}\right) 
\,.
\end{equation}
Hence, one can, without loss of generality,  choose  constants $s$ and  $K$ conveniently, therefore we assume in the sequel $s=1$ and 
$K = 1/\Gamma(1-\alpha)$, i.e. 
\begin{equation}\label{eq:Densmualfa}
	\mu_\alpha(du) =\frac{\alpha}{\Gamma(1-\alpha)}  \frac{du}{u^{\alpha+1}}.
\end{equation}

The  intensity measure of $N^{E/P}$ in \eqref{eq:NEkrozPZ} in this case can  be written as $\nu^{(2)} (du,dv) = \nu_\alpha(du) \otimes K(u,dv) $, where  $\nu_\alpha (0,u) = u^\alpha$ and $K(u,\cdot) $ is a probability marking kernel on $(0,\infty)\times \sB(0,\infty) $ corresponding to the $\tgamma(1-\alpha, u)$ distribution, see \Cref{ex:EPab}. 
This means that $N^{E/P} $ has an alternative  representation
\begin{equation}\label{eq:NEkrozPrep}
	N^{E/P} 
	= \dsum_{i=1}^\infty \delta_{Z_i, G_i/Z_i} \,,
\end{equation}
where  $\sum_i \delta_{Z_i} $  is  a Poisson process on $(0,\infty)$ with intensity measure $\nu_\alpha$ and  $(G_i)$ is an independent i.i.d. sequence of gamma$(1-\alpha,1)$ distributed random variables. 
Moreover,
\begin{equation}\label{eq:ZeindGama}
	(Z_i)_{i\geq 1} \eind (\Gamma_i^{1/\alpha})_{i\geq 1}
\end{equation} 
where $(\Gamma_i)_i$  represent the points of the unit rate  Poisson process in \eqref{eq:hPPP}. 
Therefore, by
\eqref{eq:NEkrozPrep}, the size-biased ordered sequence $(\tilde V_n)$  has the form
\begin{align}\label{eq:tildV}
	(\tilde V_n)= 
	\left(\frac{\tilde P_n}{\sum_j \tilde P_j} \right)
	= \nmap(G_n /Z_n)
	= \nmap (G_n \Pi_n),
\end{align}
where $\Pi_j =  \prod_{i=1}^{j-1} Z_i/Z_{i+1}$. 
From \eqref{eq:ZeindGama},  
it follows that   $Z_i/Z_{i+1}
= (\Gamma_i/\Gamma_{i+1})^{1/\alpha} \sim \tbeta (i \alpha,1)$ are independent random variables. 
Therefore $\sum_j \delta_{\tilde V_j}$ again belongs to beta/gamma perpetuity type of distributions.

\subsection{Poisson-Dirichlet distributions}

Pitman and Yor in \cite{PY97} proposed a change of measure in the construction  of the standard Pitman-Yor process 
to derive a more general version of  \eqref{eq:tildV}.
 We will construct the same  distribution  using an equivalent,  but  more explicit change of measure, see \eqref{eq:alfGamPoi}. Assume that $\alpha \in (0,1)$ and that a  constant $\theta$ satisfies $\theta>-\alpha$. One can now bias the distribution of $N^{E/P}$ in \eqref{eq:NEkrozPrep} by the value $Z_1^\theta = \Gamma_1^{\theta/\alpha}$. Indeed, note that 
 \[
 \label{eq:Z1Bias}
 \EE (Z_1^\theta) =  \EE (\Gamma_1^{\theta/\alpha})
 = \Gamma(\theta/\alpha+1)<\infty .
 \]
 Therefore 
 \begin{equation}\label{eq:NEPZ1Bias}
 \PP^* (\cdot) =  	\EE \left( N^{E/P} \in \cdot \,; \frac{Z_1^\theta}{ \Gamma(\theta/\alpha+1)} \right) 
 \end{equation}
 defines a new probability distribution on $M_p((0,\infty)^2)$,  see \cite{AGK19}. 
Clearly, if $\theta \not = 0$, the new probability measure $\PP^*$ changes the distribution of $\{V_n \} = \nmap(\{ P_i :T_n \leq 1\})$ in \eqref{eq:VofSeq1} and
equivalently of
$ (\tilde V_n )=   \nmap(G_n /Z_n)$ in \eqref{eq:tildV}
as well.

\begin{definition}
{\em Poisson-Dirichlet distribution} (cf. Pitman and Yor \cite{PY97})  with  parameters $\alpha \in [0,1)$ and $\theta > -\alpha$ (or PD($\alpha,\theta$)) is
defined  as follows: the case PD($0,\theta$) corresponds to the distribution 
of   $\sum_j \delta_{V_j}$ arising from the gamma subordinator
as in \eqref{eq:PDa} with $a=\theta$.
Distributions in  PD($\alpha,\theta$) class for $\alpha \in (0,1),\ \theta> -\alpha$,  correspond to the distributions 	of $\sum_j \delta_{V_j}$ in \eqref{eq:VofSeq1} after the change of measure in \eqref{eq:NEPZ1Bias}.
\end{definition}

\begin{proposition} \label{pro:GamBet1}
Each  PD($\alpha,\theta$) distribution belongs to the beta/gamma family of distributions with   independent random variables
 $G_n \sim \tgamma(1-\alpha,1)$ and $
U_n \sim \tbeta(a_n,1)$, where $a_n=\theta+ n\alpha$, $ n \geq  1$.
%
\end{proposition}

\begin{proof}
	In the case $\alpha=0$, we proved the claim after \eqref{eq:tildVa}. Assume therefore $\alpha \in(0,1)$.  It is sufficient to understand the effect of $\PP^*$ on the sequence $(\tilde V_n )$, since
	$  \{V_n \} = \nmap(\{ P_i :T_n \leq 1\})  = \{\tilde V_n \}$.
Observe that the change of measure in \eqref{eq:NEPZ1Bias}  has no influence on the sequence $(G_i)$ in \eqref{eq:NEkrozPrep}. However, under $\PP^*$  the sequence $(Z_i)$ has a new distribution, of a sequence $(Z_i')$ say. Using \eqref{eq:ZeindGama}, this distribution is straightforward to describe. 
 Now $Z'_1 ={D}_1^{\prime 1/\alpha}$ where
 $D^\prime_1 \sim \tgamma (\theta/\alpha+1,1)$ is independent of the Poisson process $(\Gamma_j-\Gamma_1)_{j\geq 2}$, while
 $Z'_j= (D'_1+ \Gamma_j - \Gamma_1)^{1/\alpha} = (D'_1+ D_2 +\cdots + D_j)^{1/\alpha}
 $ for $j\geq 2$. With this notation 
 $ \sum_i \delta_{Z'_i,G_i/Z'_i}$ has the distribution in \eqref{eq:NEPZ1Bias}. Therefore, after the change of measure, the size-biased sequence $(\tilde V_n)_n$  has  representation 
 \begin{align} \label{eq:tildVkappa}
 	\tilde V_n =   	\tilde V_n^\theta=  \frac{G_n /Z'_n}{\sum_{j=1}^\infty G_j/Z'_j}
 	= \frac{G_n \Pi_n}{\sum_{j=1}^\infty G_j \Pi_j},\quad
 	n=1,2, \ldots,
 \end{align}
 where $\Pi_j =  \prod_{i=1}^{j-1} Z'_i/Z'_{i+1}$. This time $Z'_i/Z'_{i+1}
 \sim \tbeta ( i \alpha+ \theta,1)$ as a transformation of different independent gamma distributed random variables, recovering \eqref{eq:tildV} for $\theta =0$. Therefore, \eqref{eq:tildVkappa} belongs to the beta/gamma perpetuity by definition.
 \end{proof}
 
 The proof shows that the sequence in \eqref{eq:tildVkappa}, and PD($\alpha,\theta$) distribution for $\alpha>0$ therefore, has an alternative representation
\begin{equation}\label{eq:alfGamPoi}
	 (\tilde V_n) = \nmap \left((G_n  (D'_1+ D_2 +\cdots + D_n)^{-1/\alpha}
 )\right),
\end{equation}
where $D^\prime_1 \sim \tgamma (\theta/\alpha+1,1)$ is independent of the i.i.d. standard exponential sequence $(D_i)$.

\begin{remark} \label{rem:perm}

 In \cite{PY97},   the authors
 use two other quantities to bias distribution of $\{ P_n\}$, and $\{ V_n\}$ therefore. Instead of $Z_1^\theta$ they suggest $\tau_1^{-\theta}$ and  $L^{\theta/\alpha}$,  where $L = \lim_n n V^\alpha_n$ with probability 1. However,
 	it is immediate that the three ideas yield the same distribution
 	of $( \tilde V_n)$. 
 		This follows from the simple observation that 
 		$(Z_1, \{P_i\}_i)$ have the same distribution as $(E/\tau_1, \{P_i\}_i)$ where $ E$ is an independent standard exponential random variable. 
 		The representation of $ (V_n)$ before \Cref{cor:1} and  the strong law of large numbers  together imply that
 	\[
 	n P^\alpha_n = n (V_n \tau_1)^\alpha   \asto K,
 	\]
 	 where $K>0$ is a constant in the definition of the measure $\mu_{\alpha}$. Therefore $L = K \tau_1^{-\alpha}$, hence the
 		biasing by  $L^{\theta/\alpha}$ is equivalent to  biasing by $\tau_1^{-\theta}$, see \cite{AGK19}.
 		See also Proposition 14 in \cite{PY97} and discussion thereafter.  
\end{remark}

 \subsection{Residual allocation model} 
 
Following earlier results by  McCloskey \cite{MC65} and  by Perman in  \cite{PePhd},
it was proved  in  \cite{PPY92}, that all the  distributions in the PD family have  stick-breaking representations.
We will show this as a straightforward corollary of a  theorem
concerning a larger family of 
random mass-partitions which can be represented by  beta/gamma perpetuities family of distributions.

Observe that the construction below allows any $\alpha\geq 0$, despite  the fact that for $\alpha\geq 1$ one cannot represent these distributions in a simple way using a subordinator. 

\begin{theorem}\label{thm:GPi}
	 For arbitrary constants   
	 \begin{equation}\label{eq:alphaa1c}
	 	\alpha \geq 0,\  a_1>0 \qand  c>0,
	 \end{equation}
 let $a_n = a_1+(n-1) \alpha,\ n \geq  1$. Suppose that
 $G_n \sim \tgamma(c,1)$ and $
 U_n \sim \tbeta(a_n,c+\alpha),\ n \geq  1$
 are all independent random variables.
 Then, for
 $ 
 \Pi_j = \prod_{l=1}^{j-1} U_l 
 $, for $j \geq 1$, the random series $\sum_{j=1}^\infty G_j \Pi_j$ converges a.s. and 
 $(\tilde V_n) = \nmap( G_n \Pi_n )$  has the following stick-breaking representation
 \begin{equation}\label{eq:Y1Y2gen}
 	\tilde V_1 = \tilde Y_1, \quad 	\tilde V_n= (1- \tilde Y_1)\cdots ( 1 -\tilde Y_{n-1}) \tilde Y_n, \quad n \geq 2\,,
 \end{equation}  
 with independent random variables $\tilde Y_n \sim\tbeta(c,a_n)$.
\end{theorem}

\begin{proof}
	The proof follows from \Cref{lem:gam} concerning  more general  weighted sequences of beta/gamma type. We only need  to show that sequences $a_n = a_1 + (n-1) \alpha$, $b_n\equiv \alpha+c $ and $c_n \equiv c$ satisfy its assumptions.  Assumption~\ref{ass:1}, part (i) obviously holds since $a_j+b_j -a_{j+1} = c > 0 $ for all $j$. Observe that  $\pi_j =  \frac{a_1}{a_1+b_1} 
	\cdots \frac{a_{j-1}}{a_{j-1}+b_{j-1}}$ satisfy 
\begin{align*}
\pi_j 
& = \frac{1}{1+(\alpha+c)/a_1} \frac{1+\alpha/a_1}{1+\alpha/a_1 + (\alpha+c)/a_1}
\cdots \frac{1+(j-2)\alpha/a_1}{1+ (j-2)\alpha/a_1 + (\alpha+c)/a_1} \\
 & = \frac{1}{1+y} \frac{1+ x}{1+y+ x}
\cdots \frac{1+(j-2)x}{1+ y+ (j-2)x }
= \prod_{i=0}^{j-2} \frac{1+i x}{1+y+  i x }\\
& =   
	\frac{ \Gamma(\frac{1+y}{x} )}{ \Gamma(\frac{1}{x})}
\frac{\Gamma\left( \frac{1}{x} + j-1 \right)  }{\Gamma\left(\frac{1+y}{x} + j-1\right)}
\sim C_{x,y} (j-1) ^{-y/x}\,,
\end{align*}
for $j \toi$,
where $x= \alpha/a_1 $ and $y = (\alpha+c)/a_1$, and $C_{x,y}>0$ is a constant dependent on these two parameters and where we used the asymptotic relation $\Gamma(n+z)\sim \Gamma(n) n^z$  as $n \toi$, for any real  $z$, following from the Stirling's formula.
%
Since $y >x$, this yields  (ii) and (iii) of Assumption~\ref{ass:1} at once.

\end{proof}

By \Cref{pro:GamBet1}, 
Theorem~\ref{thm:GPi}  applies to all Poisson-Dirichlet distributions, proving  that  they all have  stick-breaking representation \eqref{eq:Y1Y2gen}.
This was  first proved by McCloskey~\cite{MC65} for PD($0,\theta$), 
by Perman in \cite{PePhd} for  PD($\alpha,0$) distributions with $\alpha \in (0,1)$ , and finally by Perman, Pitman and Yor in \cite{PPY92} for general PD distributions using a method based in Markov chain theory. Bertoin in \cite{Be06book} gave an alternative proof using Palm theory.

 \begin{corollary}[McCloskey, Perman, Pitman, Yor]
  For each $\alpha \in[0,1), \theta +\alpha >0$, there exists
 a sequence of independent random variables $\tilde Y_n$ such that $\tilde Y_n \sim\tbeta(1-\alpha,n\alpha+\theta)$, and  so that PD($\alpha,\theta$) distribution has 
 the stick-breaking representation $\{\tilde V_n\}$ where
  \begin{equation}\label{eq:Y1Y2etc}
  	\tilde V_1 = \tilde Y_1, \quad 	\tilde V_n= (1- \tilde Y_1)\cdots ( 1 -\tilde Y_{n-1}) \tilde Y_n, \quad n \geq 2\,.
  \end{equation}

 \end{corollary}
 
Moreover, based on Theorem~\ref{thm:GPi} we propose the following  general class of residual allocation models.
  
 \begin{definition}
 Suppose that  for some  constants  $\alpha \geq 0,\  a_1>0$ and $  c>0$, $G_n \sim \tgamma(c,1),$ and $
 U_n \sim \tbeta(a_1+ (n-1) \alpha,c+\alpha),\ n \geq  1$,
 are all independent. Let $\Pi_j =   \prod_{l=1}^{j-1} U_l$, the random mass-partition generated by 
 \begin{align} \label{eq:tildVdef}
 	(\tilde V_n)   
 	= \nmap(G_n \Pi_n)
 \end{align}
has a stick-breaking representation as in \eqref{eq:Y1Y2gen} and we refer to it as 
the beta stick-breaking  
distribution or residual allocation model with parameters $\alpha,  a_1, c$, which we abbreviate as
 $\RAM(\alpha,  a_1, c)$. 
\COM{ which we abbreviate as $\beta$-s.b.($\alpha,  a_1, c$).}
 \end{definition}

Thus, an arbitrary PD($\alpha,\theta)$,  distribution corresponds to the $\RAM (\alpha,  \alpha+\theta, 1-\alpha)$ distribution for all $\alpha\in[0,1)$, $\theta+ \alpha >0$.

 \section{ Arcsine laws } \label{sec:PYthm}

\subsection{Mass-partitions generated by stable subordinators} \label{ssec:PYdist}

 Recall, for $\alpha \in (0,1)$, PD($\alpha,0$) distributions are derived from  the jumps of $\alpha$-stable subordinator or the points of the corresponding point process $N=\sum_{i=1}^\infty \delta_{T_i,P_i}$ which in this case 
 denotes a Poisson point process on $(0,\infty)^2$ with intensity measure 
$Leb \times \mu_\alpha$,
where $\mu_\alpha(u,\infty) =K u^{-\alpha}$ for any $u>0$, some $\alpha \in (0,1)$ and a contant $K>0$.  It will become clear that  the exact value of  the constant $K$ has no influence on any of the main results below.
The
right continuous inverse of the $\alpha$-stable subordinator $
\Sub_t = \sum_{T_i\leq t} P_i\,,$  $ t \geq 0,$ 
 is defined as a process
\begin{equation}\label{eq:Invs}
	\Inv_s = \Inv(s) = \inf\{ t : \Sub_t > s\}\,,
\end{equation}
which is known to be continuous and nondecreasing. 

The connection between the  excursions of Bessel processes and stable subordinator (and the  process $\Npp$ therefore) was  first established by Molchanov and Ostrovskii in \cite{Mo69}. For  $\alpha \in (0,1)$, the  points of $N$ correspond to the lengths of excursions of Bessel processes of dimension $d = 2 (1-\alpha)$, see \cite{PY92,PPY92,PY97}.
For Brownian motion and $\alpha =1/2$  this was already proved by L\'evy, note that  this also follows from the It\^ o's Poisson process representation of the Brownian excursions (see Le Gall \cite{LG10}).
	In that particular case, $T_i$'s and $P_i$'s above correspond to the times (on the local time scale) and the lengths of individual Brownian excursions. Moreover,  $(\Inv_s)$  corresponds to the  local time process of Brownian motion at 0, 
	see Revuz and Yor, \cite[Ch. XII]{RY99}, Proposition (2.5), while the points of increase of $(\Inv_s)$ coincide with the set of zeros of the underlying Brownian motion.

Consequently, as corollaries of the results in   \cite{PY92},  one can deduce many distributional properties of such excursions of Brownian motion, and Bessel processes in general.
Notable examples include the arcsine laws studied by L\'evy, Lamperti and  Dynkin among others. 
For instance in 1939, P. L\'evy made  the following remarkable observation:  the fraction of time Brownian motion $(B_t)$ spends positive 
has the same law
when stopped at a fixed time $T$ and when stopped at a random (inverse local) time $\Sub_s$  for which  $B_{\Sub_s} =0$.

Recall that the scaling property  \eqref{eq:ScalProp} of the Poisson process $\Npp$ implies
\eqref{eq:VofSeq1}.
Consequently, for an arbitrary $s>0$
\begin{equation}\label{eq:ppeta}
	\eta(s)  
	= \nmap \left( \{ P_i :  {T_i \leq s } \}
	\right)\,
\end{equation}
have the same random mass-partition we already met in \eqref{eq:VofSeq1}. We sometimes write $\eta =\eta(s)$.

There is an alternative way of generating a random discrete distribution from a given subordinator. One can consider the partition arising from the closure of the  image of the subordinator  intersected with a closed interval $[0,T]$ rescaled by $T>0$. 
The points in this  image split $[0,T]$ into countably many smaller intervals. Note that, by this construction, the last of these intervals has the length 
\[
A_T = T- \Sub({\Inv_T-}).
\] 
We refer to $A_T$ as the undershoot of the level $T$ by the subordinator $(\Sub_s)$. Note that in the Brownian interpretation, $A_T$ corresponds to the age of excursion in progress at time $T$. We introduce  also the overshoot of the subordinator as
\[
B_T = \Sub({\Inv_T}) - T.
\]

Applying \eqref{eq:ScalProp} with $c =T^{ -\alpha}$,
 for a fixed number $T>0$,  we obtain 
\begin{align*}
 \left(
\dsum_{i }
\delta_{T_i,P_i}; 
\Inv_1 \right) & \eind \left(
 \dsum_{i }
 \delta_{\frac{T_i}{T^\alpha},\frac{P_i}{T}} ; 
\inf\{ s : \sum_{T_i/T^\alpha \leq s}
P_i /T > 1 \}\right)
 \\
 & = \left(
\dsum_{i }
\delta_{\frac{T_i}{T^\alpha},\frac{P_i}{T}} ; 
\frac{1}{T^\alpha} \inf\{ s :  \sum_{T_i \leq s}
P_i > T  \}\right)=
\left(
\dsum_{i }
\delta_{\frac{T_i}{T^\alpha},\frac{P_i}{T}} ; 
\frac{1}{T^\alpha} \Inv_T \right).
\end{align*} 
Observe that
\[ 
1- \sum_{T_i < \Inv_T } {{P_i}/{T}}  = A_T/T \qand   \sum_{T_i \leq \Inv_T } {{P_i}/{T}} -1  = B_T/T,
\]
therefore   
\begin{equation}\label{eq:EqualDist1}
	\left(
	\dsum_{T_i <\Inv_1 }
	\delta_{P_i}; 
A_1,B_1
	\right)
	\eind
	\left( \dsum_{T_i < \Inv_T } \delta_{{P_i}/{T}}; 
A_T/T, B_T/T
	\right).
\end{equation}

This motivates  the following definition
\begin{equation}\label{eq:ppxi}
\xi = \xi(T) =  
\{{P_i}/{T} :\  {T_i < \Inv_T } \} \cup \{ 
A_T/T \}.
\end{equation}
Note that the points on the right hand side add up to one, so no self-normalization $\nmap$ is needed in this case.
Moreover, by  \eqref{eq:EqualDist1}, $\xi(T)$ has the same distribution for each $T>0$, and by the same argument, it follows that the distributions of  $\eta(s)$ and $\xi(T)$  do not depend on the constant $K$ in $\mu_\alpha(u,\infty) = K u^{-\alpha}$ either.  Therefore, 
for convenience we  set $K=  \Gamma(1-\alpha)^{-1}$ as before.

The generalized arcsine law of Pitman and Yor in \cite{PY92} claims that  
$\xi$ and $\eta$ have the same distribution.
In particular, the random distributions $\xi$ and $\eta$ both 
correspond to the Pitman-Yor distribution with parameter $\alpha$ introduced in \eqref{eq:VofSeq1}.
Although both mass-partitions  are  generated from the jumps $P_i$ of the same subordinator $(\Sub_s)$, their equality in distribution appears quite surprising at the first glance, since $\eta$ incorporates only completed jumps, while $\xi$ includes additional value  $A_T/T$   corresponding to the incomplete jump over a predetermined fixed threshold $T$.
Another  surprising fact, observed in \cite{PY92}, is that $A_T/T$ can be seen as a size-biased sample from the points $\{V_i\}$ of the process $\eta$.  More precisely, if conditionally on $\eta= \sum_i \delta_{V_i}  =\sum_{T_i \leq 1 } \delta_{P_i/ \Sub_1} $ we select $J \in \{ 1,2, \ldots \}$ at random with probability $V_j$,  denoting $\tilde V_1= V_J$ it follows that
\begin{equation*}\label{eq:X0YJ}
A_T/T \eind \tilde V_1\,.
\end{equation*}

\subsection{Thinning} \label{ssec:Thin}
Recall, the exponential marking  of 
the process $\Npp= \sum_i \delta_{T_i,P_i}$ in \eqref{eq:NnaE}, produced a new Poisson process
$	N^E = \dsum_i  \delta_{T_i,P_i,E_i}$.
From, $N^E$, one can derive two thinnings of $\Npp$
\begin{align}
N^\wedge& 
= \dsum_i \delta_{T^\wedge_i,P^\wedge_i}  =
\dsum_i \delta_{T_i,P_i} \1{P_i>E_i} \qand  \nonumber \\
N^\vee & = \dsum_i \delta_{T^\vee_i,P^\vee_i}  =\dsum_i \delta_{T_i,P_i} \1{P_i\leq E_i}.\label{eq:Nvee}
\end{align}

It is immediate, see \cite[Corollary 5.9]{LP18}, that $N^\wedge$ and $N^\vee$ are two independent Poisson processes on $(0,\infty)^2$. Moreover, the corresponding intensity measures are product of the Lebesgue measure with measures
\[ 
\nu^\wedge (A) = \nu^\wedge_\alpha (A) = \int_A \mu_\alpha(dx) \int_{s< x} e^{-s} ds =
\int_A \mu_\alpha(dx)(1- e^{-x}),
\]
\[	
\nu^\vee (A) = \nu^\vee_\alpha (A) =\int_A \mu_\alpha(dx) \int_{s \geq x} e^{-s} ds =
\int_A \mu_\alpha(dx)e^{-x},
\]
respectively. Furthermore,
\[ 
N^{\wedge,E}  =\sum_j \delta_{T^\wedge_j,P^\wedge_j,E^\wedge_j} = \dsum_i \delta_{T_i,P_i,E_i} \1{P_i>E_i}
\]
is also a Poisson process with the intensity measure $\nu^{\wedge,E}$ on $(0,\infty)^3$, where
\[ 
\nu^{\wedge,E} (D\times C_1 \times C_2 ) = |D| \int_{C_1} \mu_\alpha(dx) \int_{s\in (0,x) \cap C_2 } e^{-s} ds 
= |D| \, \mu_\alpha^{(2)}( C_1 \times C_2 )
\]
where $|D|$ denotes the Lebesgue measure of  set $D$ and $\mu_\alpha^{(2)}$ denotes  a probability distribution (by the choice of $K$) on $(0,\infty)^2$ given by $\mu_\alpha^{(2)} ( C_1 \times C_2 ) =
\int_{C_2 } e^{-s} ds  \mu_\alpha(C_1 \cap (s,\infty))$.
It follows that  $N^{\wedge,E}$ can be viewed as a unit rate homogeneous Poisson process $T^\wedge_1 \leq T^\wedge_2 \leq \ldots$, 
whose points are  independently marked with i.i.d. points $(P^\wedge_j,E^\wedge_j) \in (0,\infty)^2$ with the law $\mu_\alpha^{(2)}$.
Moreover, 
\[ 
(P^\wedge_j,E^\wedge_j)_j \eind ( E^\wedge_j /U_j^{1/\alpha},E^\wedge_j)_j 
 \]
where  $(E^\wedge_j )$ is an i.i.d. sequence of $\tgamma (1-\alpha,1)$ distributed random variables independent
of an i.i.d. sequence $(U_j)$ from uniform distribution on the interval $(0,1)$.
In particular, the first arrival time
$
T^\wedge_1 
$
has the standard exponential distribution  and is independent of $E^\wedge_1$.

 For a fixed $s>0$, the spatial component of $N^\vee$
 up to the time $s$  denoted by
 \[ 
	N^\vee_s =  \sum_{T^\vee_i \leq s} \delta_{P_i^\vee} 
  \]
  has a simple representation in terms of  
  the unit rate homogeneous Poisson process   $\{\Gamma_i\}$   and an independent i.i.d. sequence $(G_n)$ of $\tgamma(1-\alpha,1)$ distributed random variables. Indeed, denote 
\begin{equation*}\label{eq:ppMs}
	 M^\vee_s =	\dsum_{i \geq 1} \delta_{G_i (s/(s+\Gamma_i))^{1/\alpha} }  \,.
\end{equation*}
By \eqref{eq:NEkrozPrep}, this can be viewed as the distribution of
$ \sum_{i\geq 2}
 \delta_{\tilde P_i} $ conditionally on $Z_1^\alpha = s$, although we do no use this fact  in the sequel. 
  
 \begin{lemma}\label{lem:NveeRep}
 	For  any $\alpha\in (0,1)$, the point processes $	N^\vee_s \sim \PPP(s \cdot \nu^\vee) $
and $ M^\vee_s$ above satisfy
\begin{equation}\label{eq:NveeRep}
	N^\vee_s \eind M^\vee_s. 
\end{equation}
 \end{lemma}
\begin{proof}
	The point process $M^\vee_s$  is clearly a Poisson process. Observe that
\begin{align*}
 \EE  M_s(u,\infty) = 
  s\,   \EE \left( \left( \frac{G_1}{u} \right)^\alpha -1  \right)_+ =
  s\,    \left( \frac{u^{-\alpha}}{\Gamma(1-\alpha)} e^{-u} - \PP(G_1 > u) \right) 
  = s \nu^\vee( u,\infty),
\end{align*}
where we use  partial integration in the last equality. Thus $M^\vee_s$  and $N^\vee_s$ are two Poisson processes with equal intensity measure.
\end{proof}

\subsection{Pitman-Yor theorem}

\begin{theorem}[Pitman and Yor]\label{thm:1}
	Assume $\alpha\in (0,1)$, for all constants $K, \, T, \, s >0$,
	 the point processes 
	 in \eqref{eq:ppeta} and \eqref{eq:ppxi} have the standard Pitman-Yor distribution, in particular
	\begin{equation}\label{eq:xieta}
		\xi(T) 
\eind \eta(s)\,.
	\end{equation}
	
\end{theorem}

\begin{proof}
	
	Recall that  $\Npp = \sum_i \delta_{T_i,P_i} = N^\vee + N^\wedge $.  We observed above that  $T^\wedge_1$ has the standard exponential distribution. Denote now
	\begin{equation}\label{eq:Gamma}
		\Gamma = \sum_{T^\vee_i \leq T^\wedge_1} P_i^\vee + E_1^\wedge\,.
	\end{equation} 
	By the definition  $T^\wedge_1 = \Inv_\Gamma $ almost surely, see \eqref{eq:Invs}.
	It turns out that $\Gamma $ itself is again a standard exponential random variable. This might be obvious in the light of the memoryless property of the exponential distribution of $E_i$'s. And moreover, random variable $\Gamma$ is also independent of $\Npp$, see \Cref{lem:2}.
	Thus,
 by \eqref{eq:EqualDist1} 
 it suffices to show  $\xi(\Gamma) $ has the same distribution as $ \eta(1)$. 
 
Observe that
\begin{equation}\label{eq:AGamBGam}
	 A_\Gamma = E_1^\wedge \eind G_1 \sim \tgamma(1-\alpha,1)
\end{equation}
is independent of $T^\wedge_1$ and $N^\vee$.

Assume $(G_n)$ is an i.i.d. sequence from $ \tgamma(1-\alpha,1)$ distribution,  and that $(\Gamma_i)$  represent the points of the unit rate homogeneous Poisson process as in \eqref{eq:tildV}. Recall $	N^\vee_{T^\wedge_1} =  
\sum_{T^\vee_i \leq T^\wedge_1} \delta_{P_i^\vee}$, hence, by \Cref{lem:NveeRep}
\begin{equation}\label{eq:Nveeref}
	  N^\vee_{T^\wedge_1} \eind  \left\{ 
   G_n \left(  \frac{\Gamma_1}{\Gamma_1+\Gamma_n-\Gamma_1}  \right) ^{1/\alpha}
  \right\}_{n\geq 2}\,,
\end{equation}
since both sides can be seen as Cox processes, that is Poisson processes mixed over  equally distributed random values, i.e. $T^\wedge_1$ and $\Gamma_1$  
respectively. Thus, in the notation of \eqref{eq:NEkrozPrep}
\begin{equation}\label{eq:ANwedgeRem}
	 \left(  N^\vee_{T^\wedge_1} ; T^\wedge_1,A_\Gamma \right)
 \eind 
 \left(
  \left\{ G_n \left( \frac{\Gamma_1}{\Gamma_n}  \right) ^{1/\alpha}
  \right\}_{n\geq 2}  ;  \Gamma_1,G_1\right)
  =\left(
  \left\{ G_n \frac{Z_1}{Z_n} 
  \right\}_{n\geq 2}  ;  Z_1^{\alpha},G_1 \right)
\end{equation}
with the last component independent of the other two on  both sides.
Therefore, by \eqref{eq:tildV}
\begin{equation*}
	\xi(\Gamma) = \nmap \left(  
	N^\vee_{T^\wedge_1} \cup \{ A_\Gamma\}
	\right)
	\eind \nmap\left( \left\{ G_n
	\frac{Z_1}{Z_n}
	\right\}_{n\geq 2} \cup \{ G_1 \} \right) =
 \nmap(G_n /Z_n) = \eta(1).
\end{equation*}

\end{proof}

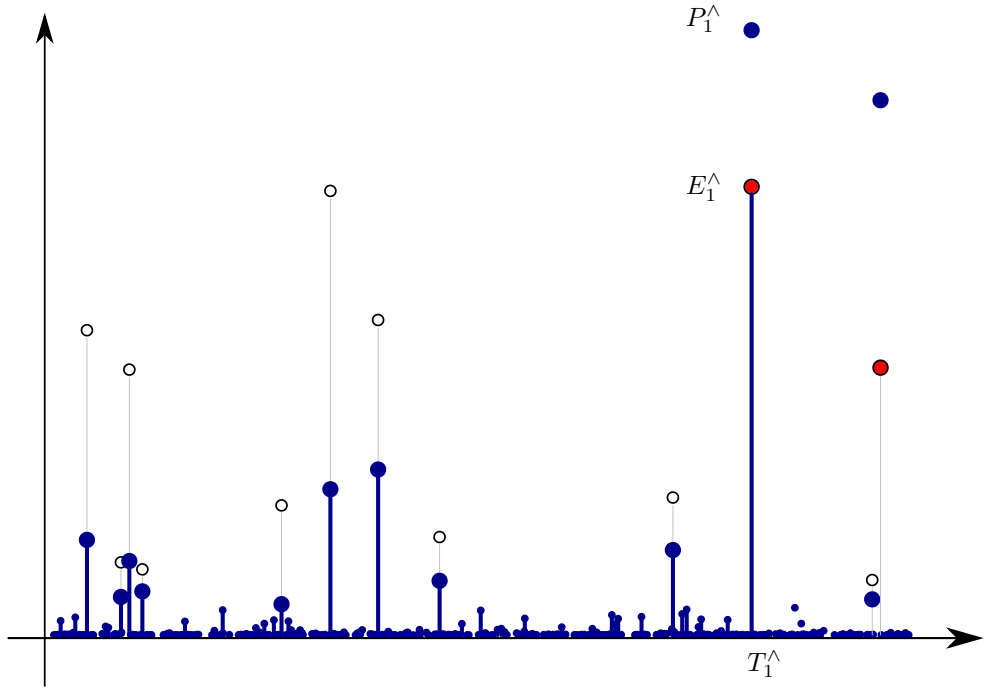
\begin{figure}[!h]
	\centering
	\setlength{\unitlength}{0.1\textwidth}
	\begin{picture}(10,9)
		\put(0,0){\includegraphics[width=\textwidth]{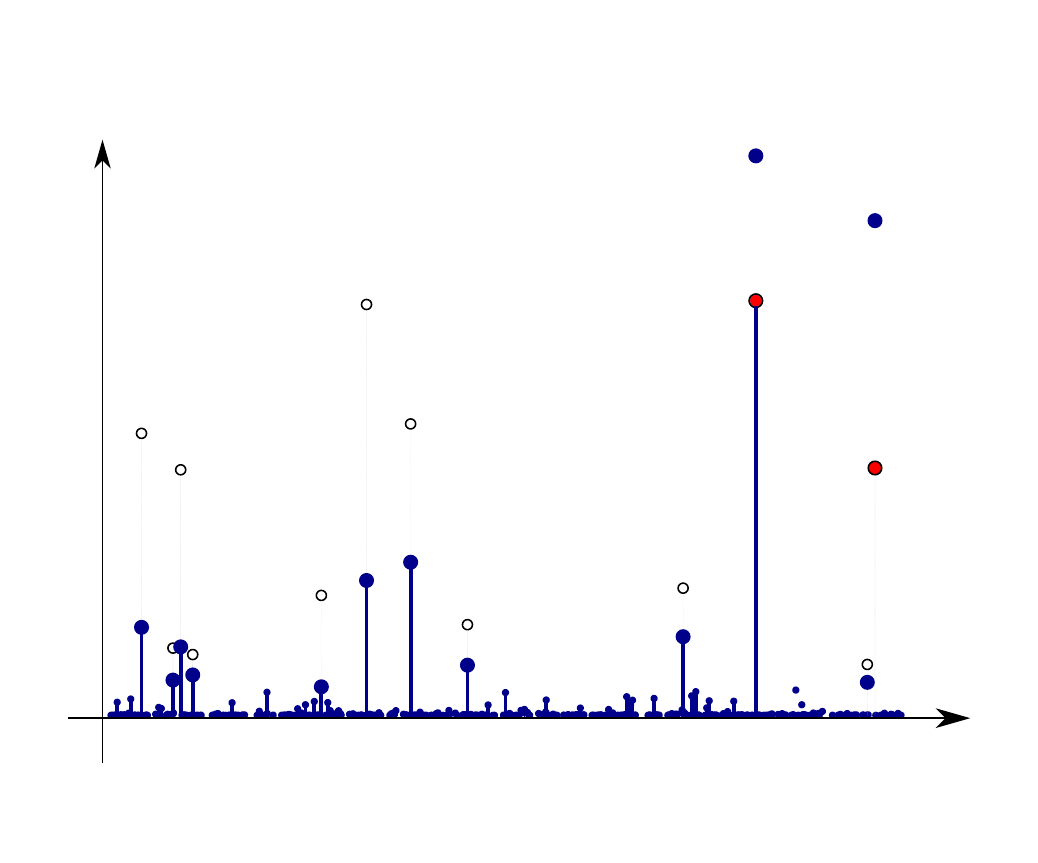}}
		\put(6.7,6.85){$P^\wedge_1$}
		\put(6.7,5.35){$E^\wedge_1$}
		\put(7.25,1.1){$T^\wedge_1$}
	\end{picture}
\vspace{-2.2cm}
	\caption{An illustration of the points of $N^\vee$ and $N^\wedge$ in the proof of \Cref{thm:1} (in blue).
	By other circles  (in white and red) we denote corresponding exponential marks for  the points above a certain threshold.
	Note that for points of  $N^\wedge$, their exponential mark in red falls below their respective blue point.  Random variable $\Gamma$ in \eqref{eq:Gamma} corresponds to the total length of vertical blue lines.}
	\label{fig:PYproof}
\end{figure}

The relation \eqref{eq:ANwedgeRem} in the proof seems    remarkable in its own right. 
Observe that the undershoot and overshoot random variables 	$A_\Gamma$ and $B_\Gamma$ satisfy
\begin{equation}\label{eq:ABGam}
	 A_\Gamma + B_\Gamma =P_1^\wedge = E_1^\wedge/U_1^{1/\alpha}.
\end{equation}

\begin{theorem} \label{cor:AGammaV}
 For each $\alpha \in (0,1)$, for the size biased ordering 
 $(\tilde V_n)$ of the Pitman-Yor  process in \Cref{thm:1}
 and an independent $U\sim \tunif(0,1)$, 
  it holds that 
\begin{equation}\label{eq:ANnormV}
\left( \frac{A_T}{T},\sum_{T_i < L_T} \delta_{P_i/T} , \frac{B_T}{T} \right) 
 \eind	\left( \frac{A_\Gamma}{\Gamma}  ,
	\frac{N^\vee_{T^\wedge_1}}{\Gamma},\frac{B_\Gamma}{\Gamma} \right)
	\eind \left(\tilde V_1,  \{ \tilde V_j \}_{j\geq 2}, \tilde V_1 \left( \frac{1}{ U^{1/\alpha}} -1 \right)\right)
\end{equation}		
in the notation of \eqref{eq:ppxi}.
\end{theorem}

\begin{proof}
	The first equality in distribution follows from \eqref{eq:EqualDist1} and the proof of the previous theorem.
 Directly from \eqref{eq:tildV} and \eqref{eq:ANwedgeRem} it follows that  
 \begin{equation}\label{eq:ANnormP}
 	\left( \frac{A_\Gamma}{T^{\wedge 1/\alpha}}  ,
 	\frac{N^\vee_{T^\wedge_1}}{T^{\wedge 1/\alpha}},
 	\frac{B_\Gamma}{T^{\wedge 1/\alpha}} \right)
 	\eind \left(\tilde P_1,   \{ \tilde P_j \}_{j\geq 2},\tilde P_1 \left( \frac{1}{ U^{1/\alpha}} -1 \right) \right).
 \end{equation}	
Applying the self-normalization by the sum of points in  two first components on both sides yields the claim.
\end{proof}

	Pitman and Yor in \cite{PY92} immediately gave two proofs of \Cref{thm:1}.
	The first of them relies on a special identity for Palm distributions of the corresponding  Poisson processes (a similar argument appears in \cite{PPY92} as well).  	Later,  Pitman gave a combinatorial  argument for the theorem, see \cite{Pi97} as well as Theorem 4.7 and Exercise 4.4.1 in \cite{Pi06book}. 
	The thinning proof above is related to the  proof in Section 3 of \cite{PY92}, however the geometric approach 
	suggested here is entirely constructive and avoids analysis of  conditional distributions of appropriate point processes. On top of this, it has a number of simple corollaries completing the picture of the excursions of Bessel processes as we show next.

\subsection{Excursions of Bessel processes}

 Assume that $(X_s)$ is a Bessel process of dimension $d= 2 (1-\alpha) \in (0,2)$. With $(L_s)$ representing the local time of  the Bessel process at 0, the lengths of its excursions away from 0 correspond to the jumps of
 the inverse local time process $\Sub_t =\inf\{s:L_s >t \}$, $t \geq 0$. 
 As we observed earlier, see also \cite{Mo69,PPY92,Pi06book},  $(\Sub_t)$ is an
  $\alpha$-stable subordinator which can be represented as
 \[ 
 \Sub_t = \sum_{T_i \leq t} P_i, \quad t\geq 0.
 \] 
 One can define the left and the right end point of the excursion straddling a fixed time point $T$ as
\begin{gather*}
	g_T= \sup \{ s < T : X_s = 0 \} \qand
	d_T= \inf \{ s > T : X_s = 0 \}.
\end{gather*} 
Both values can be expressed in terms of the
corresponding subordinator as
\begin{gather*}
	g_T=  \Sub(L_{T} -) \qand
	d_T= \Sub(L_{T})\,,
\end{gather*}
see (1.6) in \cite{Be99} for instance. 
In this notation $A_T = T- g_T$ and $B_T = d_T - T$.
Observe that $\Delta_T= A_T+B_T$ corresponds to the total length of the excursion straddling time $T$.
In the case $\alpha=1/2$, corresponding to Brownian motion, the
joint distribution of $(g_T,d_T)$ is well known. But it was also studied for $T=\bex$, where $\bex$ denotes an independent unit exponential random variable. In Winkel~\cite{Wi05}, one can find the Laplace transform of the following seven-dimensional  random vector
\begin{equation}\label{eq:7-tup}
 (\bex, L_{\bex}, B_{\bex}, A_{\bex}, g_{\bex} , d_{\bex},
\Delta_{\bex}) =	 (\bex, L_{\bex}, B_{\bex}, A_{\bex}, \tau(L_{\bex} -), \tau(L_{\bex} ),
 \Delta_{\bex}),
\end{equation}
even in a more general case of subordinator $(\tau_t)$ associated with a general Bernstein function $\Phi$, so that
$\EE e^{-\lambda \tau _t} = e^{- t \Phi(\lambda )}$. In the case of Bessel  process $(X_t)$ this random vector has a surprisingly simple structure.

\begin{proposition} \label{thm:7tup}
	Suppose that random variables $G_1\sim \tgamma(1-\alpha,1)$, $U\sim \tunif(0,1)$, $T_1^\wedge \sim \texp(1)$ and point process
	$N^\vee$  as in Subsection \ref{ssec:Thin} are all independent. Then the random vector in \eqref{eq:7-tup}
	has the same distribution as
	\[ 
	\left(
	\Gamma,
	T_1^\wedge, G_1 (U^{-1/\alpha}-1),G_1,
	G_0,  G_0 + G_1 U^{-1/\alpha},
	G_1 U^{-1/\alpha}
	\right)
	\]
	where $G_0 =  \sum_{T^\vee_i \leq T_1^\wedge} P_i^\vee\sim  \tgamma(\alpha,1)$ and
	$\Gamma =G_0+G_1$. 
\end{proposition}

\begin{proof}
	In the proof of \Cref{thm:1} simply substitute $\Gamma$ with $\bex$  and recall that $\Delta_\Gamma= A_\Gamma+ B_\Gamma = G_1/U_1^{1/\alpha}$ where $G_1= E^\wedge_1 \sim \tgamma(1-\alpha,1)$  and that 
	$G_0 = \sum_{T^\vee_i \leq T^\wedge_1} P^\vee_i  =g_\Gamma$ is independent of it. Since $G_0+G_1 = \Gamma$, it follows that $G_0 \sim \tgamma(\alpha,1)$. In the end, note
 that $d_\Gamma = g_\Gamma + \Delta_\Gamma$.
	
\end{proof}


The one dimensional equality in distribution  $\Delta_{\bex} \eind  G_1 U^{-1/\alpha} $ represents the key finding of Theorem 1.1 in
Bertoin, Fujita, Roynette and Yor  \cite{BFRY06}. The authors recognize 
 that this distribution can be seen as a perpetuity 
 itself
  and derive several other distributions related to it.
  The distribution of $\Delta_{\bex}$ was called BFRY law
 by Devroye and James in \cite{DJ14}  who also discuss other BFRY distributions.
   BFRY distribution was later extended to derive a new  class of random discrete distributions, see  Ayed et al. \cite{ALC19}.
   
 
Equipped  with \Cref{thm:7tup},  one can deduce a result concerning the excursion of the   Bessel process of dimension $d\in (0,2)$ straddling an arbitrary fixed time $T$.  Recall that the subordinator $(\Sub_t)$ tracks excursions of $(X_s)$ away from 0 and that by \Cref{cor:AGammaV}
 \[
 \sum_{T_i < L_T} \delta_{P_i/T}  
 \eind  \{ \tilde V_n \}_{n\geq 2}.
 \]
 Moreover, if we apply self-normalization functional $\nmap$ to the right hand side above,
 we obtain random mass-partition
 \begin{equation}\label{eq:etaPrime}
 	 \eta' = \nmap \left(\{ \tilde V_n \}_{n\geq 2} \right)
 	 = 
 	 \frac{1}{1-\tilde V_1} \{ \tilde V_n \}_{n\geq 2}
 	 =
 	 \left\{ 
 	 \prod_{j=2}^{n-1}( 1 -\tilde Y_{j}) \tilde Y_n \right\}_{n \geq 2},
 \end{equation}
where $1-\tilde V_1=1-\tilde Y_1$ is independent of 
$\eta'$ because  $ (\tilde Y_n)$ form an independent sequence distributed as 
in \eqref{eq:Y1Y2etc} for $\theta=0$.  By \eqref{eq:etaPrime},
$\eta'$ has PD($\alpha,\alpha$)  or   $\RAM(\alpha,2\alpha,1-\alpha)$ distribution, see also Section 6.4 of \cite{PY97}.

 \begin{theorem} \label{cor:6tup}
 	Assume $(X_s)$ is a Bessel process of dimension $d= 2 (1-\alpha) \in (0,2)$, for any $T>0$
  \begin{align}\label{eq:6-tup}
	\MoveEqLeft{ \left(	\frac{A_{T}}{T},\frac{ B_{T}}{T}, \frac{g_{T}}{T} , \frac{d_{T}}{T},
		\frac{\Delta_{T}}{T},\sum_{T_i < L_T} \delta_{P_i/T} \right) }   \\
	& \eind \left(  1-Q, (1-Q)(U^{-1/\alpha} -1), Q,
	Q+ (1-Q)U^{-1/\alpha} ,(1-Q)U^{-1/\alpha},
	Q \eta'
	\right) \nonumber
\end{align}
where $Q \sim \tbeta(\alpha,1-\alpha)$,
 $U \sim \tunif(0,1)$ and
 $\eta'\sim$PD$(\alpha,\alpha)$   on the right hand side are independent.
 \end{theorem}

\begin{proof}
Recall that the unit exponential random variable $\Gamma$ in \eqref{eq:Gamma} (cf. also
	\Cref{cor:AGammaV} and \Cref{thm:7tup}) is independent of
	the subordinator $(\Sub_t)$.  By \eqref{eq:EqualDist1},
	it suffices to show the claim for $T=\Gamma$.
	Observe that $A_\Gamma$ and $B_\Gamma$ of \eqref{eq:ABGam} are independent of $N^\vee_{T^\wedge_1}$. As in the proof of \Cref{thm:7tup} denote by 
	$G_0= S^\vee_{T^\wedge_1} =  \sum_{T^\vee_i \leq {T^\wedge_1}} P_i^\vee$ and $G_1=E^\wedge_1$. Therefore
	$\Gamma = G_0+G_1$ by definition. By the independence between $G_0$ and $G_1$, it follows that 
	\[
	  Q = \frac{G_0}{\Gamma} \sim \tbeta(\alpha,1-\alpha)\,.
	\]
After observing that 
	$g_\Gamma =  S^\vee_{ T_1^\wedge} = G_0$ and that
	$d_\Gamma =g_\Gamma +  A_\Gamma + B_\Gamma$, it remains to note that 
\[
\frac{N^\vee_{T^\wedge_1}}{\Gamma} \eind \{ \tilde V_j \}_{j\geq 2} = (1-\tilde V_1)  \eta' = Q \eta'\,,
\]
where $Q=1-\tilde V_1 = (\Gamma -G_1)/\Gamma = G_0/\Gamma$
is independent of $\eta'$ 
 in \eqref{eq:etaPrime}.
  
	
\end{proof}

For $\alpha =1/2$, \Cref{cor:6tup} gives a 6-tuple arcsine law for excursions of Brownian motion until and straddling  time  $T>0$. In this particular case, an easy direct calculation gives the well known expression for the joint density of $(g_T/T,d_T/T)$, see \cite{YY13}, as 
 \[ 
   f_{g_1,d_1}(x,y)  = \frac{1}{2 \pi } x^{-1/2}  (y-x)^{-3/2},
   \quad x \in (0,1), \  y >1.
  \]
Moreover, in the case of Brownian motion, the random variables $Q$, and $1-Q$  have  arcsine distribution, so one can view \Cref{cor:6tup} as a generalization of the first arcsine law.
 
 
 Concerning the last component of 6-tuple law of \Cref{cor:6tup}, it 
is known that the normalized values $\eta' = \nmap(\{ \tilde V_j \}_{j\geq 2})$ correspond to the excursions of the $d$-dimensional Bessel bridge, $d \in (0,2)$, see \cite{PPY92} or \cite{PY97}.
 
  The second classical arcsine law for Brownian motion concerns the occupation time of the positive half line,  i.e. 
\[ 
 O^+_T = \dint_0^T \1{X_s>0} ds\,.
 \] 
 It is known that the individual excursions have independent and random signs, therefore for an i.i.d. Bernoulli sequence $(\vep_i)$ taking values  $0$ and $1$ with equal probability, we have 
\[ 
 \frac{O^+_T}{T} \eind \sum_{i\geq 1} V_i \vep_i = 
 \sum_{i\geq 1} \tilde V_i \vep_i = (1-Q) \vep_1 + \sum_{i\geq 2}\tilde V_i \vep_i\,.
 \]

 Thus, if we introduce a random operator  $ \vep(Q,\eta') = (1-Q) \vep_1 + \sum_{i\geq 2}\tilde V_i \vep_i$, one obtains
 more general arcsine law.
 
 \begin{corollary} \label{cor:BM7law}
 Suppose $(X_s)$ denotes a standard Brownian motion, then for any $T>0$
  \begin{align*}\label{eq:7b-tup}
	\MoveEqLeft{ \left(	\frac{O^+_T}{T}, \frac{A_{T}}{T},\frac{ B_{T}}{T}, \frac{g_{T}}{T} , \frac{d_{T}}{T},
		\frac{\Delta_{T}}{T},\sum_{T_i < L_T} \delta_{P_i/T}
		  \right) }   \\
	& \eind \left(\vep(Q,\eta'),  1-Q, (1-Q)(U^{-2} -1), Q,
	Q+ (1-Q)U^{-2} ,(1-Q)U^{-2},
	Q \eta'
	\right), \nonumber
\end{align*} 
where
$Q \sim \tbeta(1/2,1/2)$,
$U \sim \tunif(0,1)$,
$\eta'\sim$PD$(1/2,1/2)$    and an i.i.d. Bernoulli sequence $(\vep_n)$ on the right hand side  are independent.
 \end{corollary}


  It is a classical result by L\'evy  that 
  random variables  $Q$ and $\vep(Q,\eta)$ above both have arcsine distributions, see \cite{Le39}. Extensive literature exists concerning the occupation times for Bessel processes in general, see \cite{Pi2018} and references therein.
 \Cref{cor:BM7law}  can be generalized  to 
  skewed Brownian motions or 
 other Bessel processes, see \cite{Pi2018} again. 
  We do not see a simple way of incorporating  the third arcsine law concerning the time when the maximum of $(X_s)$ is attained, i.e. $h_T = \mbox{argmax}_{s \leq T} X_s$,  into a statement like \eqref{eq:6-tup}.
  Surprisingly, even in the case of Brownian motion, the joint bivariate distributions of $g_T/T$, $O^+_T /T$ and $h_T/T$ have been calculated by analytical methods only recently in \cite{HM25}.

\section{Representations of residual allocation models} \label{sec:Repr}

\subsection{Representations in the case $\alpha =0$}

\Cref{cor:AGammaV} relates PD($\alpha,0$) distribution with the Poisson process $N^\vee$ in \eqref{eq:Nvee}. Somewhat surprisingly,
PD($0,\theta$) process, i.e. $\RAM(0,\theta,1)$ distribution has a representation through a  Poisson process $N^\vee$  if we let $\alpha$ tend to 0 and rescale time accordingly. 

\begin{proposition}\label{prop:alfa0P}
	For an arbitrary  $\theta>0$,   the  random mass-partitions
	$\nmap( N^\vee_{\theta/\alpha})$ and  $\nmap( M^\vee_{\theta/\alpha})$ (see \Cref{lem:NveeRep}) converge in distribution, as $\alpha\to 0$, to the PD$(0,\theta)$ distribution.
\end{proposition}
\begin{proof}
	 By \eqref{eq:Densmualfa}, the intensity measure of a Poisson process $N^\vee_{\theta/\alpha}$ has the form
\[ 
 \frac{\theta}{\Gamma(1-\alpha)}
 \frac{dx}{x^{\alpha+1}} e^{-x}.
 \]
It is clearly converging to $\gamma(dx)= \theta x^{-1} e^{-x} dx $ as $\alpha \to 0$, which was the intensity measure of the Poisson process $N$  such that $\nmap(N)$ has the PD($0,\theta$) distribution. One is tempted to conclude that 
 $\nmap( N^\vee_{\theta/\alpha}) \dto \nmap( N)$ at once. However, functional $\nmap$ is nowhere continuous on $M_{p,f}$, and additional effort is needed to overcome this.
 
 Consider the summation functional $s:M_{p,f} \to [0,\infty)$ given by
 \begin{equation*}\label{eq:sfunct}
 	\sum_i \delta_{x_i} \mapsto 
 	s \left(\sum_i \delta_{x_i}\right)  = \sum_i x_i \,.
\end{equation*}
Note, the functional $s$ is not continuous either. Take an arbitrary nonnegative, bounded and continuous  function 
$f:(0,\infty) \to [0,\infty)$ with a support in
$(r,\infty)$ for some $r>0$ and an arbitrary $\lambda>0$. Observe that, as $\alpha \to 0$
\begin{align*}
\EE e^{ -\lambda	s ( N^\vee_{\theta/\alpha}) - 
	N^\vee_{\theta/\alpha} f }
	  & =\exp \left( 
	 - \int_0^\infty (1- e^{-\lambda x- f(x)}) e^{-x}
	 \frac{\theta}{\Gamma(1-\alpha)} \frac{dx}{x^{\alpha-1}} \right) \\
	 & \to 
	  \exp \left( 
	 - \int_0^\infty (1- e^{-\lambda x - f(x)}) \theta e^{-x}
	 \frac{dx}{x} \right)= \EE e^{ -\lambda	s ( N) - 
	 	N f }\,,
\end{align*}
where we denote
$(\sum_i \delta_{x_i}) f=\sum_i f(x_i)$ for any $\sum_i \delta_{x_i} \in M_p$.

By an adaptation of \cite[Theorem 5.30]{Ka21}, it follows that
\[
\left( s ( N^\vee_{\theta/\alpha}), N^\vee_{\theta/\alpha} \right) \dto
\left( s( N), N \right)
\]
jointly in $(0,\infty) \times M_p(0,\infty)$, since the scaling operator $(t,\sum_i \delta_{x_i}) 
\mapsto t \sum_i \delta_{x_i} := \sum_i \delta_{t x_i }$ is known to be continuous, we obtain
\[ 
 \nmap ( N^\vee_{\theta/\alpha}) \dto
 \nmap(N).
 \]
\end{proof}

A related and slightly more complicated representation of the PD$(0,\theta)$ distribution  appears  in \cite{PY97} following Proposition 21.
\COM{maybe no cont. of $\nmap$ needed PD's are cont!?}

\paragraph{Connection with Dickman's subordinator}  

Apart from gamma and $\alpha$-stable subordinator, the literature on random discrete distributions often considers Dickman's subordinator  (see \cite{IMS21} and references therein). This subordinator arises as in \eqref{eq:GenSub}  from a Poisson process on $(0,\infty) \times (0,1)$
$$
\Npp = \dsum_i \delta_{T_i,P_i} \sim \PPP(Leb \times  x^{-1}  dx).
$$
Clearly
$$
N_a = \dsum_{T_i \leq a} \delta_{P_i}
\sim \PPP(a \cdot x^{-1}  dx) 
\eind \dsum_{i} \delta_{\exp (-\Gamma_i/a)}
$$
where as before  $(\Gamma_i)$ denote the points of the standard homogenous Poisson process.
Consider now, the partition of the closed interval $[0,1]$ produced by the points $P_i$ (thus not by the subordinator). 
If we set $\Gamma_0=0$,  the intervals in this partition can be listed (starting from the top of the segment $[0,1]$) as follows 
$$
e^{-\Gamma_{n-1}/a} - e^{-\Gamma_n/a}
=  Y_n \prod_{i=1}^{n-1} (1-Y_i),\quad n =1,2,\ldots ,
$$
where we note
$ Y_i = 1 -  e^{-D_i/a} \sim \tbeta (1,a) $ for $i \geq 1$. Recognize that this corresponds to the RAM($0,a,1)$ and therefore to
the standard Poisson Dirichlet distribution  PD($a,0$).
This is actually a reformulation of Ignatov's result in \cite{Ig82}.
However, due to the  representation \eqref{eq:tildV}, we obtain the following unexpected identity.
\begin{corollary} \label{cor:DickIde}
	Assume $G_n$ is an i.i.d. sequence from  $\texp(1)$ distribution
	independent of the homogeneous Poisson process $(\Gamma_i)$, then with $\Gamma_0=0$
	$$
	\left(e^{-\Gamma_{n-1}/a} - e^{-\Gamma_n/a} \right)_n
	\eind 
	\left(  \frac{G_n /\exp( \Gamma_n/a) }{\sum_{j=1}^\infty G_j/\exp(\Gamma_j/a )} \right)_n.
	$$
	
\end{corollary}

 Recall, PD($0,\theta$) or  $\RAM(0,\theta,1)$ distributions  were  represented in \eqref{eq:PDGamPoi} as  $\nmap((G_n e^{ -(D_1+\cdots + D_n)/\theta}))$
 for i.i.d. standard exponentials $(D_i)$ and $(G_i)$.
 With  additional biasing, one can represent $\RAM(0,a,c)$ for arbitrary $a>0$ and $c>0$ in a similar form. Consider an i.i.d. sequence 
$(\tilde D_i)$ with a density 
\begin{equation}\label{eq:tildDdens}
	  \frac{1}{K_{a,c}} (1- e^{- x /a})^{c-1} e^{-x}\,, 
\end{equation}
on $(0,\infty)$, 
 where $K_{a,c} = a \Gamma(a) \Gamma(c) /\Gamma(a+c)$.
It can be viewed as a sequence $D_i$, $i =1,2,\ldots, $ with each element biased by $(1- e^{- D_i /a})^{c-1} $.

\begin{proposition}
	Fix arbitrary   constants  $a>0$ and $  c>0$. Suppose that   $(G_n)$ is an i.i.d. sequence from $ \tgamma(c,1)$ distribution and that elements of an independent i.i.d. sequence $(\tilde D_i)$ have the density in \eqref{eq:tildDdens}. Then 
	\begin{equation*} %
	 (\tilde V_n) = 	\nmap \left(G_n e^{ -(\tilde D_1+\cdots + \tilde D_n)/a}
		\right),	
	\end{equation*}
	is a representation of the $\RAM(0,a,c)$ distribution.
\end{proposition}

\begin{proof}

 It is sufficient to note that
\[ 
\nmap \left(G_n e^{ -(\tilde D_1+\cdots + \tilde D_n)/a}
\right) =\nmap \left(G_n \Pi_n \right),
 \] 
 where $\Pi_n = \prod_{i=1}^{n-1} U_i =
 \prod_{i=1}^{n-1} \exp( -  \tilde D_{i+1}/a)$.
It is not difficult to see that i.i.d. sequence $(\exp( -  \tilde D_{i}/a))_i$ has $\tbeta(a,c)$ density. 
Since $G_n \sim \tgamma(c,1)$, by definition $ (\tilde V_n) $ corresponds to the
$\RAM(0,a,c)$ distribution.

\end{proof}

\subsection{Representations in the case  $\alpha \in (0,1)$}

Consider the standard Pitman-Yor   distribution, which corresponds to  PD($\alpha,0$) or  $\RAM(\alpha,\alpha,1-\alpha)$ distribution, for $\alpha \in (0,1)$.
It was constructed by  the self-normalization of  the points of 
Poisson process $\tsum_{T_i\leq 1} \delta_{P_i}$ 
with intensity measure satisfying $\mu_\alpha(u,\infty) = K u^{-\alpha}$ for $u>0$.
Assume they
are ranked in nonincreasing order,  so that $P_1\geq P_2 \geq P_3 \ldots$. Comparing this process with 
Poisson process $\tsum_i \delta_{(\Gamma_i/K)^{-1/\alpha}}$, where $(\Gamma_i)$ represents   Poisson process in \eqref{eq:hPPP}, it is immediate that they have
the same intensity measure. Hence, the sequence $(P_i)$ has the distribution of $((\Gamma_i/K)^{-1/\alpha})$ . Recall, the successive ratios
$ Y_n = (\Gamma_{n-1}/\Gamma_{n})^{1/\alpha} $  are independent with  $\tbeta((n-1)\alpha,1)$ distribution for $n \geq 2$. 
Since $V_j = P_j / \tau_1$  where $\tau_1= \tsum_i P_i$, using
$Y_n = P_n/P_{n-1}$, the following  simple representation of the Pitman-Yor process is immediate
(and known, see formula (23) in \cite{PY97}).

\begin{proposition}\label{cor:1}
	The nonincreasing representation $( V_j)$ of PD($\alpha,0$) distribution $\nmap(\tsum_{T_i\leq 1} \delta_{P_i})$ has the following representation
	\begin{equation}\label{eq:VnRepr}
		V_j = \frac{P_j}{\sum_{i=1}^\infty P_i}
		= V_1  Y_1 \cdots Y_{j-1}\,, \quad  j \geq 1, 
	\end{equation} 
	where $(Y_n)$ is a sequence of independent random variables such that $ Y_n \sim\tbeta((n-1)\alpha,1)$
	and $V_1 = 
	\left({ \tsum_{i\geq 1} \prod_{i=1}^{i-1} Y_{i}} \right)^{-1}$.
\end{proposition}

Recall that
\Cref{cor:AGammaV} provides another simple representation of the Pitman-Yor process, i.e. PD($\alpha,0$) distribution, using a  random variable $A_\Gamma\sim \tgamma(1-\alpha,1)$ and an independent Poisson process $N^\vee$. Interestingly  for
all $\alpha\in(0,1)$ and  $\theta>0$, PD($\alpha,\theta$) has a representation in terms of a mixed version of $N^\vee$ or $M^\vee$.

\begin{proposition} \cite[Proposition 21]{PY97}
	\label{prop:RepAlfa}
	Fix arbitrary constants $\alpha\in (0,1)$ and $ \theta>0$, assume $\tilde D\sim \tgamma(\theta/\alpha)$ is independent of the Poisson point process $N^\vee$.
	Then
	$
	\nmap (N^\vee_{\tilde D})
	$ and
	$
	\nmap (M^\vee_{\tilde D})
	$
	have
	PD($\alpha,\theta$) 
	distribution.
\end{proposition}

\begin{proof}
	
	From \Cref{lem:NveeRep} we obtain 
	\begin{equation}\label{eq:DugaNN}
		\frac{N^\vee_s}{s^{1/\alpha}} 
		\eind \frac{M^\vee_s}{s^{1/\alpha}}\,. 
	\end{equation}
	The same holds if we substitute $s$ by an independent random
	variable $\tilde D\sim \tgamma(\theta/\alpha,1)$. 
	As before, let $(D_i)$ denote an i.i.d. sequence of \texp(1)  random variables. Denote by  $D' = \tilde D + D_1 \sim \tgamma(\theta/\alpha+1,1)$.
	Observe that
	\[ 
	\frac{M^\vee_{\tilde D}}{\tilde D^{1/\alpha}}
	= \dsum_{i \geq 1} \delta_{G_i (\tilde D+\Gamma_i)^{-1/\alpha} }=
	\dsum_{i \geq 1} \delta_{G_i (D'_1+D_2+\cdots + D_i)^{-1/\alpha} }
	\]
	According to \eqref{eq:alfGamPoi}, it follows that
	$\nmap(N^\vee_{\tilde D}) $ and $\nmap(M^\vee_{\tilde D})$
	have the PD$(\alpha,\theta)$ distribution. 
\end{proof}

\begin{remark}
In the case $\alpha \geq 1$, it is an open problem to find an alternative representation  of general $\RAM(\alpha,a_1,c)$ distributions using a (mixed) Poisson process analogous to \eqref{eq:PDGamPoi} or  \eqref{eq:alfGamPoi} for instance.
Pitman \cite{Pi2018} points that even PD$(\alpha,\theta)$  family with $\theta<0$   might have a counterpart to \Cref{prop:RepAlfa}.

Furthermore, PD($\alpha,\theta$) distributions arise as a natural limit of the Chinese restaurant process \cite{Pi06book} or Hoppe's urn process \cite{Cr16}, it would be interesting to construct a modification of these processes which in the limit approaches $\RAM(\alpha,a_1,c)$ distribution.
\end{remark}

\section{
Beta-gamma algebra for infinite sequences} \label{sec:3}

We recall the basic facts of the beta-gamma algebra, see Dufresne \cite{Du10}.
 For all  $a,b, \theta>0$, it is well known that if random variables  $A\sim \tgamma(a,\theta)$ and $B\sim \tgamma(b,\theta)$  are independent, then $U = A/(A+B) $ is $ \tbeta(a,b)$ distributed and  independent of $V= A+B \sim \tgamma(a+b, \theta)$. Moreover, if $U$ and $V$ are independent with distribution given above, their product $A = UV $ has $\tgamma(a,\theta)$ distribution. Finally, if 
$V$ and $U \sim \tbeta(a,b)$ are independent and 
$ UV \sim \tgamma(a,\theta)$, it follows that $V\sim \tgamma(a+b, \theta)$. The last claim follows by the fact that log beta distributions are infinitely divisible \cite{Du10}, and therefore have non vanishing characteristic functions.

 In this section we show that parts of beta-gamma algebra extend to sequences of random variables. Such an inductive beta-gamma algebra  is the basis of all the stick-breaking representations in this article. 

\begin{hypothesis}\label{ass:1}
	Assume that strictly positive sequences $(a_n)$ and $(b_n)$ 
	satisfy
	\begin{enumerate}[(i)]
		\item $a_j+b_j -a_{j+1} > 0 $ for each $j \geq 1$\,,
		\item 
		$
		\sum_{j=1}^\infty \pi_{j+1} (a_{j}+b_{j} -a_{j+1}) < \infty\,,
		$
		\item 
		$
		\pi_j a_j  \to 0 \,,
		$ as $j \toi$.
	\end{enumerate} 
	where  $\pi_j =  \frac{a_1}{a_1+b_1} 
	\cdots \frac{a_{j-1}}{a_{j-1}+b_{j-1}}$ $j\geq 1$,  with convention $\pi_1 =1$.
\end{hypothesis}


\begin{theorem} \label{lem:gam}
	Assume sequences $(a_n)$, $(b_n)$ in $(0,\infty)$
	satisfy Assumption~\ref{ass:1}. Fix a constant $c_1>0$,  and let 
	$c_j = a_{j-1}+b_{j-1} -a_{j}$ for $j \geq 2$. 
	Suppose that $G_n \sim \tgamma(c_n,\theta),$ and $
	U_n \sim \tbeta(a_{n},b_n),\ n \geq  1$
	are all independent random variables.	
	Denote 
	$ 
	\Pi_j = \prod_{l=1}^{j-1} U_l 
	$, for $j \geq 1$. 
	Then the random series $\sum_{j=1}^\infty G_j \Pi_j$ converges a.s. Moreover, 
	\begin{align*} 
		\tilde V_n 
		= \frac{G_n \Pi_n}{\sum_{j=1}^\infty G_j \Pi_j},\quad
		n=1,2, \ldots,
	\end{align*}
	have the following representation
	\begin{equation*}
		\tilde V_1 = \tilde Y_1, \quad 	\tilde V_n= (1- \tilde Y_1)\cdots ( 1 -\tilde Y_{n-1}) \tilde Y_n, \quad n \geq 2\,.
	\end{equation*}  
	with $\tilde Y_n \sim\tbeta(c_n,a_n)$ independent.
\end{theorem}

\begin{proof}

	Consider the sequence,
\begin{equation}\label{eq:Wn}
	W_{n} = \sum_{j\geq n+1} G_j \prod_{l=n}^{j-1} U_l  \,, \quad n \geq 1,
\end{equation}	
the key step of the proof involves establishing that $W_n$ are well defined gamma distributed random variables.
	By assumption~\ref{ass:1}, $\EE W_1 < \infty$, therefore 
the random series for $W_1$ converges a.s.
Suppose that $W_1 \sim \tgamma(a_1,\theta)$.
By straightforward induction, it follows $W_n \sim \tgamma(a_n,\theta)$ for each $n \geq 1$. Indeed, assume this holds for $W_{n-1}$, observe that for any $n \geq 2$ 
\begin{equation}\label{eq:WUGW}
	 W_{n-1}  = U_{n-1} (G_n + W_n)
\end{equation}
 with $U_n$ and $G_n$ and $ W_n$ all independent. By beta-gamma algebra, it follows that $G_n+ W_n$  has $\tgamma(a_{n-1}+b_{n-1},\theta)$ distribution, and therefore
 $W_n\sim \tgamma (a_n,\theta)$ as claimed.

To show that $W_1 \sim \tgamma(a_1,\theta)$, observe  
\[ W_1 =  G_2 U_1+ \cdots + G_{n-1} U_1 \cdots U_{n-2}
+(G_n +W_n) U_1 \cdots U_{n-1}.
\] 
If one could assume  $W_n\sim \tgamma (a_n,\theta)$, by the independence of $G_n,\ U_n, \ n \geq 1$	and by reversing the argument after \eqref{eq:WUGW}, it follows that $W_1 \sim \tgamma(a_1,\theta)$. 
To show this, we couple a sequence $(W_n)$ with a sequence  which
has the desired distribution. 
Consider an auxiliary sequence 
$W'_{n}\sim \tgamma (a_n,\theta)$ independent of $G_n,\ U_n, \ n \geq 1$ above. Let  
\[
W_{n,1} = G_2 U_1+ \cdots + G_{n-1} U_1 \cdots U_{n-2}
+(G_n +W'_n) U_1 \cdots U_{n-1}.
\]
As $n \toi$, it satisfies
$\EE | W_{1} -W_{n,1} |  \to 0$ as $n \toi$. But now, as explained above, $W_{n,1}\sim \tgamma(a_1,\theta)$ for each $n$. Hence, the same holds for $W_1$. By the first part of the proof, $W_n$ has the same distribution as $W'_n$.
 
To prove \eqref{eq:Y1Y2gen}, observe first that
$
\tilde Y_n = G_n/( G_n+W_n)
$
have $ \tbeta(c_n,a_n)$ distribution. 
Moreover, by \eqref{eq:WUGW}, one can write
\begin{align*}
	\MoveEqLeft{  \frac{G_n \Pi_n}{\sum_{j=1}^\infty G_j \Pi_j} 	= \frac{G_n \prod_{l=1}^{n-1} U_l   }{ G_1+W_1}  =
		\frac{G_n}{ G_n+W_n } 
		\dfrac{(G_n+W_n) U_{n-1}}{ G_{n-1}+W_{n-1} }
		\cdots
		\dfrac{(G_2+W_2) U_1}{ G_1+W_1 }}  \\
	& = \tilde  Y_n 	\dfrac{W_{n-1}}{ G_{n-1}+W_{n-1} }
	\cdots
	\dfrac{W_1}{ G_1+W_1}			
	=\tilde  Y_n (1-\tilde Y_{n-1}) \cdots (1-\tilde Y_1) \,.
\end{align*}
 To show mutual independence of $\tilde Y_i$'s, for simplicity consider $\tilde Y_1$ and $\tilde Y_n = G_n /(G_n + W_n)$, observe that the latter is independent of $G_n + W_n$ and that 
\[ 
\tilde Y_1  = 
\dfrac{G_1 }{G_1+W_1}=
G_1/
\left(G_1 
+G_2 U_1+ \cdots
+(G_n +W_n) U_1 \cdots U_{n-1}\right)\,,
 \]
with all the terms on the right hand side independent of $\tilde Y_n$. By analogous calculations, it follows that 
$\tilde Y_n$ is independent of each of $\tilde Y_1,\ldots, \tilde Y_{n-1}$.

\end{proof}

\section{Notation and auxiliary results}\label{Sec:aux}

\paragraph{Space of counting measures.}
On a general Polish space $\bbS$ we denote by $\sS$ a family of so called bounded measurable sets. 
We call set $B \subseteq \bbS$ bounded if for some fixed point $s_0 \in \bbS$ and a fixed metric generating chosen topology on $\bbS$, there exists a ball centred at $s_0$ containing such a set.   By  $M_p(\bbS) = M_p(\bbS,\sS)$ we denote the space of counting measures on $\bbS$ which are finite on bounded sets in $\sS$. Note, one can find a metric on $(0,\infty)$ which makes sets bounded when they are contained in $(\vep,\infty)$ for some $\vep >0$. Similarly, there exists a metric which preserves the standard topology on $(0,\infty)^2$ and which makes a set bounded if it is contained in $(0,1/\vep) \times (\vep,\infty)$ for some $\vep >0$.  These metrics can be chosen so that they preserve the standard topology on $(0,\infty)$ and $(0,\infty)^2$, see \cite{BP19}.  They precisely define the concepts of boundedness  we use in
the space  $M_p(0,\infty)$ and  $M_p((0,\infty)^2)$ respectively.
Following Kallenberg \cite{Ka17}, the space $M_p(\bbS)$ is endowed with the vague topology
which makes $M_p(\bbS)$ Polish again and generates 
corresponding Borel $\sigma$-algebra on $M_p(\bbS)$ that used throughout the text.

Recall, a point process on an underlying space $\XX$ is simply a random element in $M_p(\XX)$. Poisson point processes as the main example have distributions determined by their intensity measure, we refer to \cite{Ka17} or \cite{LP18} for a general theory. We often use  transformations of a Poisson process to arrive at new Poisson process, see \cite[Theorem 5.1]{LP18} for instance. Suppose that $
N^E = \sum_i  \delta_{T_i,P_i,E_i}$  is a Poisson process with intensity measure equal to $Leb \times \mu \times \nu_{E}$ where $\nu_E$ denotes the law of the standard exponential distribution. 
For a fixed $a>0$, in \eqref{eq:NEkrozPgen} we transformed  $N^E$ 
into
$N^{E/P} = N^{E/P}_a = \sum_{T_i \leq a} \delta_{E_i/P_i, P_i}$,
which is again Poisson point process. Its intensity measure, $\nu^{(2)}$ depends on $\mu$ and $a$ and can be calculated as in Theorem 4.5 in \cite{PPY92}, but for some examples it is straightforward to calculate.

\begin{example} \label{ex:EPab}
(a) Assume $\mu(dt) = \gamma(dt) = t^{-1} e^{-t} dt $. as in the case of gamma subordinator.
Consider a set of the form $D_{u,v}= (0,u) \times (v,\infty)$ with arbitrary $u,v>0$ and observe that
\begin{align*}
	\nu^{(2)}(D_{u,v}) = \EE\left[N^{E/P}_a(D_{u,v}) \right] & = \int_v^\infty  \frac{a}{t} e^{-t} dt
	\int_{0}^{ut} e^{-s} ds =
	\int_{0}^{u}\frac{a\, ds}{1+s } \int_v^\infty   e^{-(1+s)t} (1+s) dt.
\end{align*}
Thus, one can write 
$\nu^{(2)} (du,dv) = \gamma_a (du) \otimes K(u,dv) $,  
where $\gamma_a (0,u) = \log (1+u)^a$ and $K(u,\cdot) $ is a probability marking kernel on $(0,\infty)\times \sB(0,\infty) $ corresponding to the  exponential distribution with parameter $1+u$.

(b) 
 Assume $\mu(dt) = \mu_{\alpha}(dt) $ as in \eqref{eq:Densmualfa}.
Consider again a set of the form $D_{u,v}$ for $u,v>0$ as above. Observe that
\begin{align*}
	\nu^{(2)}(D_{u,v}) = \EE\left[N^{E/P}(D_{u,v}) \right] & = \int_v^\infty \mu_\alpha(ds) \int_{0}^{us} e^{-t} dt =
	\int_{0}^{u}\alpha  t^{\alpha-1} dt \int_v^\infty \frac{s^{-\alpha} }{\Gamma(1-\alpha) }  e^{-st}  t^{1-\alpha} ds .
\end{align*}
In this case therefore, 
$\nu^{(2)} (du,dv) = \nu_\alpha(du) \otimes K(u,dv) $, where   $\nu_\alpha (0,u) = u^\alpha$ and $K(u,\cdot) $ is a probability marking kernel on $(0,\infty)\times \sB(0,\infty) $ corresponding to the $\tgamma(1-\alpha, u)$ distribution.
\end{example}

 We denote by   $M_p(0,\infty) $  the set of counting measures on $(0,\infty)$ which are finite on sets of the form $(u,\infty)$,  $u>0$,
 and   further by  $M_{p,f}= \{ \mu \in M_p(0,\infty) : \int_{(0,\infty)} x \mu(dx) <\infty  \}$
 and  $M_{p,1}= \{ \mu \in M_p(0,\infty) : \int_{(0,\infty)} x \mu(dx) =  1 \}$, the sets of counting measures  whose points sum up to a finite number and 1, respectively.     Observe that $\sum_{T_i \leq T} \delta_{P_i}$ a.s. belongs to
 $M_{p,f}$ for any fixed $T$.

Denote now by $M_{\text{simp}}= M_{\text{simp}}((0,\infty)^2) $ the set of all counting measures $m = \sum_i \delta_{t_i,p_i} \in M_p$
for which  the projection $ \sum_i \delta_{t_i}\1{p_i>\vep}$  is a simple counting measure finite on time interval $ (0,s)$ for each $\vep >0$ and $s>0$.
Denote further
\begin{align*}
	\MoveEqLeft{A_p = \left\{m = \sum_i \delta_{t_i,p_i} \in M_{\text{simp}} :
		\ \sum_{t_i<s} p_i <\infty \mbox{ for all } s>0, \mbox{ and } 
		\right.}\\
	& \qquad \left.  
	\sum_i p_i \1{p_i>\vep} = \infty \mbox{ for some } \vep >0 \mbox{ small enough} \right\}.
\end{align*}

Take now an arbitrary $m= \sum_i \delta_{t_i,p_i} \in A_p$, and an i.i.d. sequence $(E_i)$ of standard exponential random variables, and let
\begin{equation}\label{eq:wtT}
	\tilde T = \inf \{ t_i : p_i >E_i\}.
\end{equation}

\begin{lemma}\label{lem:1}
	For any $m= \sum_i \delta_{t_i,p_i} \in A_p$, the infimum in \eqref{eq:wtT} is attained. In particular
	$\tilde T = \min \{ t_i : p_i >E_i\}$ for a unique value $t_j$, and if we denote by $\tilde E $ the corresponding $E_j$, then the random variable
	\[
	\tilde \Gamma  = \sum_{t_i<\tilde T} p_i + \tilde E \mbox{ has the standard exponential distribution}.
	\]
\end{lemma}

\begin{proof}
	If one uses the image of the function  $g(s)= \sum_{t_i \leq s} p_i$  to partition the half line $[0,\infty)$ and the memoryless property of the exponential distribution, the claim is wholly unsurprising, since we could not find a reference to this fact we give a detailed proof. 
	
	Step 1.
	Observe first that the point process $\sum_i \delta_{t_i,p_i} \1{E_i < p_i}$ has finite intensity $\sum_{t_i < s} (1- e^{-p_i}) \leq \sum_{t_i < s} p_i$
	on each set $[0,s)\times[0,\infty)$. Moreover, it is a.s. nontrivial process, because $\PP(E_i \geq p_i \mbox{ for all } i) = \exp(- \sum_i p_i) = 0$. 
	
	Step 2. Assume now that $(q_j)$ is a sequence in $(0,\infty)$, satisfying $\sum_i q_i =\infty$, and denote its partial sums by $s_k = \sum_{j=1}^{k} q_j $.
	Recall that $(E_j)$ is an i.i.d. sequence of standard exponential random variables. Let
	$ J=  \min\{ j: E_j < q_j\}$, then $\PP(J <\infty) =1$ by the argument in the previous step. Hence, 
	$\Gamma = s_{J-1} + E_J$ is well defined. Take now $u>0$, then $u \in (s_{k-1},s_k]$, for some $k \geq 1$. Clearly
	$\PP(\Gamma>u) = \PP(E_i\geq q_i\,, i =1, ..., k-1; \, E_k > u -s_{k-1} )
	= e^{s_{k-1}} e^{-u +s_{k-1}} = e^{-u}$. Thus $\Gamma \sim \texp(1)$.

	Step 3.  Consider now point process $\sum_i\delta_{t_i,p_i,E_i}$ with the state space $(0,\infty)^3$.
	Since $m \in A_p$, we know that $\sum_i\delta_{t_i,p_i, E_i}\1{p_i > \vep}$ for any $\vep>0$  can be written as $\sum_j\delta_{t^\vep_j,p^\vep_j, E^\vep_j}$ with $0\leq t^\vep_1 < t^\vep_2 < \ldots$
	Consider now random value $J^\vep= \min \{j: p^\vep_j > E^\vep_j \}$. It is well defined and a.s. finite by the argument given  in Step 1. 
	Let $ T^\vep  = t^\vep_J$ and $ \Gamma^\vep = \sum_{t^\vep_j < T^\vep} p^\vep_j + E^\vep_J= \sum_{t_i < T^\vep} p_i \1{p_i > \vep} + E^\vep_J$. By the step 2, $\Gamma^\vep \sim \texp(1)$.
	
	Step 4. By construction $\tilde T \leq T^\vep$, however $\PP(E_i >0, \mbox{ for all } i) = 1$, thus with probability one there exists $\vep >0 $  such that
	$\tilde T = T^\vep$, and in that case $\tilde E = E^\vep_J$. Finally $\sum_{t_i<s} p_i\1{p_i > \vep} \to \sum_{t_i<s} p_i $ by the monotone convergence argument, therefore $\Gamma^\vep \asto \tilde \Gamma$ as $\vep \to 0$, so
	$\tilde \Gamma$ is standard exponential as claimed.
	
\end{proof}

\begin{lemma}\label{lem:2}
	The point process $\Npp= \sum_i \delta_{T_i,P_i}$ is independent of the  random variable  $\Gamma$ in \eqref{eq:Gamma} which has the standard exponential distribution.
\end{lemma}	

\begin{proof}
	
	Observe that  $\Npp $ belongs to the set $ A_p$ of the lemma above with probability one. Therefore,  for any $u >0$ and measurable set in $M_p((0,\infty)^2)$
	\begin{align}\label{eq:DFLT}
		\PP \left( \Npp \in A ; \Gamma \leq u  \right)
		=
		\EE  \left[\1{\Npp \in A} \EE \left( \1{ \Gamma \leq u} \mid\, \Npp \right) \right] =
		\left( 1- e^{-u} \right) \PP  \left( \Npp \in A \right).
	\end{align}
	
\end{proof}



\bibliography{pitmanyor}
\bibliographystyle{abbrv}

\end{document}